\documentclass[]{article}  
\usepackage{amssymb}
\usepackage{amsmath}
\usepackage{float}
\usepackage{subfigure}
\usepackage{caption}
\usepackage{tikz}
\usetikzlibrary{decorations.markings}

\usepackage{hyperref}
\hypersetup{hypertex=true,
	colorlinks=true,
	linkcolor=blue,
	anchorcolor=blue,
	citecolor=blue}

\usetikzlibrary{backgrounds}
\newtheorem{thm}{Theorem}[section]
\newtheorem{conj}{Conjecture}[section]
\newtheorem{coro}{Corollary}
\newtheorem{definition}{Definition}
\newtheorem{lem}[thm]{Lemma}

\newtheorem{prop}[thm]{Proposition}
\newtheorem{obse}[thm]{Observation}
\newtheorem{case}{Case}

\newtheorem{problem}{Problem}

\newcommand{\yw}[1]{{#1}}
\usepackage{amsfonts}
\newenvironment{proof}{{\medbreak\noindent\it  Proof.}\,}{\hfill$\square$\medbreak}

\usepackage{geometry}
\usepackage{makecell}

\usetikzlibrary{shapes, arrows, calc, arrows.meta, fit, positioning} 
\tikzset{  
	-stealth,auto,node distance =0.8 cm and 1 cm, thin, 
	state/.style ={circle, draw, inner sep=0.2pt}, 
	point/.style = {circle, draw, inner sep=0.18cm, fill, node contents={}},  
	el/.style = {inner sep=2pt, align=right, sloped}  
}

\begin{document}
	
	\title{Strong arc decompositions of split digraphs\thanks{Research supported by the Independent Research Foundation of Denmark under grant number DFF 7014-00037B}}
	\author{J. Bang-Jensen\thanks{Department of Mathematics and Computer Science, University of Southern Denmark, Odense Denmark. (email: jbj@imada.sdu.dk)}\and Y. Wang \thanks{School of Mathematics, Shandong University, Jinan 250100, China and Department of Mathematics and Computer Science, University of Southern Denmark, Odense Denmark. (email:yunwang@imada.sdu.dk,wangyun\_sdu@163.com)}}
	
	\maketitle
	
	\begin{abstract}
          A {\bf strong arc decomposition} of a digraph $D=(V,A)$ is a partition of its arc set $A$ into two sets $A_1,A_2$ such that the digraph $D_i=(V,A_i)$ is strong for $i=1,2$. Bang-Jensen and Yeo (2004) conjectured that there is some $K$ such that every $K$-arc-strong digraph has a strong arc decomposition. They also proved that with one exception on 4 vertices every 2-arc-strong semicomplete digraph has a strong arc decomposition. Bang-Jensen and Huang (2010) extended this result to locally semicomplete digraphs by proving that every 2-arc-strong locally semicomplete digraph which is not the square of an even cycle has a strong arc decomposition. This implies that every 3-arc-strong locally semicomplete digraph has a strong arc decomposition. A {\bf split digraph} is a digraph whose underlying undirected graph is a split graph, meaning that its vertices can be partioned into a clique and an independent set. Equivalently, a split digraph is any digraph which can be   obtained from a semicomplete digraph $D=(V,A)$ by adding a new set $V'$ of vertices and some arcs between $V'$ and $V$. In this paper we prove that every 3-arc-strong split digraph has a strong arc decomposition which can be found in polynomial time and we provide infinite classes of 2-strong split digraphs with no strong arc decomposition. We also pose a number of open problems on split digraphs.
		
	\end{abstract}
	\noindent{\bf Keywords:} \texttt{Split digraph; semicomplete digraph, strong arc decomposition, branchings}

	\section{Introduction}\label{intro}
	Notation follows \cite{bang2009} so we only repeat a few definitions here (see also Section \ref{sec-prelim}).  A digraph is not allowed to have parallel arcs or loops. A directed multigraph can have parallel arcs but no loops.  	A directed multigraph is {\bf semicomplete} if it has no pair of non-adjacent vertices.

	A directed multigraph $D$ is {\bf strong} if there exists a path from $x$ to $y$  in $D$ for every ordered pair of distinct vertices $x$, $y$ of $D$ and $D$ is {\bf $k$-arc-strong} if $D\setminus{}A^{\prime}$ is strong for every subset $A^{\prime} \subseteq A(D)$ of size at most $k - 1$.  A \textbf{strong arc decomposition} of a directed multigraph $D = (V, A)$ is a decomposition of its arc set $A$ into two subsets $A_1$ and $A_2$ such that $A_1\cap A_2=\emptyset$ and  both of the spanning subdigraphs $D_1 = (V, A_1)$ and $D_2 = (V, A_2)$ are strong. Note that a directed multigraph with a strong arc decomposition must be 2‐arc‐strong.

        \begin{thm}\cite{bangTCS438}
          It is NP-complete to decide whether a digraph has a strong arc decomposition.
        \end{thm}

        In fact it was shown in \cite{bangTCS438} that the problem  is already NP-complete for 2-regular digraphs.

        An {\bf out-branching} (resp., {\bf in-branching}) of $D$ is a spanning oriented tree in which every vertex except one, called the {\bf root}, has in-degree (resp., out-degree) one in $D$. Thomassen \cite{thomassen1989}  made the following difficult conjecture.

        \begin{conj}
          \label{conj:CTbranch}
          There exists an integer $K$ such that every $K$-arc-strong digraph $D=(V,A)$   has an out-branching rooted at $u$ which is arc-disjoint from some in-branching rooted at $v$ for every choice of $u,v\in V$.
          \end{conj}

Clearly, a strong digraph contains an out-branching (resp., an in-branching) with arbitrary given root so the following conjecture by Bang-Jensen and Yeo would imply Conjecture \ref{conj:CTbranch}.
	
        \begin{conj}\cite{bangC24}
          There exists an integer $K$ such that every $K$-arc-strong digraph has a strong arc decomposition.
	  \end{conj}

	Bang-Jensen and Yeo   gave a characterization of semicomplete digraphs with a strong arc decomposition. This characterization implies that every 3-arc-strong semicomplete digraph has a strong arc decomposition.
	
\begin{thm}\cite{bangC24}\label{thm-SD}
	A 2‐arc‐strong semicomplete digraph $D$ has a strong arc decomposition if and only if $D$ is not isomorphic to the digraph $S_4$ depicted in Figure \ref{fig-DMwithoutAD}. Furthermore, a strong arc decomposition of $D$ can be obtained in polynomial time when it exists.
\end{thm}	

Bang‐Jensen and Huang extended Theorem \ref{thm-SD} to a super class of semicomplete digraphs, namely  locally semicomplete digraphs. A digraph is \textbf{locally semicomplete} if every two vertices with a common out‐ or in‐neighbor have an arc between them. The \textbf{square} of a directed cycle $v_1v_2\cdots v_nv_1$ is obtained by adding an arc from $v_i$ to $v_{i+2}$ for every $i \in [n]$, where $v_{n+1} = v_1$ and $v_{n+2} = v_2$. Note that $S_4$ above is the square of a 4-cycle.

\begin{thm}\cite{bangJCT102}
	A 2‐arc‐strong locally semicomplete digraph $D$ has a strong arc decomposition if and only if $D$ is not the square of an even cycle. Every 3-arc-strong locally semicomplete digraph has a strong arc decomposition and such a decomposition can be obtained in polynomial time.
\end{thm}

Let $D$ be a digraph with vertex set $\{v_i: i\in[n]\}$, and let $H_1,\ldots,H_n$ be digraphs which are pairwise vertex-disjoint. The \textbf{composition} $D[H_1,\ldots,H_n]$ is the digraph $Q$ with vertex set $V(H_1)\cup\cdots\cup V(H_n)$ and arc set $(\bigcup_{i=1}^nA(H_i))\cup\{h_ih_j:h_i\in V(H_i),h_j\in V(H_j), v_iv_j\in A(D)\}$. A composition $D[H_1,\ldots,H_n]$ is called \textbf{semicomplete composition} if $D$ is semicomplete.

Sun, Gutin and Ai \cite{sunDM342} gave  a characterization of a subset of  semicomplete compositions with a strong arc decomposition.  Later, Bang-Jensen, Gutin and Yeo solved the problem for all semicomplete compositions as follows, here $C_n, K_n$ and $P_n$ denote the  directed cycle, the complete digraph and the directed path on $n$ vertices, respectively.

\begin{thm}\cite{bangJGT95}
	Let T be a strong semicomplete digraph on $t \geq 2$ vertices and
	let $H_1,\ldots, H_t$ be arbitrary digraphs. Then $Q = T[H_1,\ldots, H_t]$ has a strong arc decomposition if and only if $D$ is 2‐arc‐strong and is not isomorphic to one of the following
	four digraphs: $S_4$, $C_3[\overline{K}_2, \overline{K}_2, \overline{K}_2]$, $C_3[\overline{K}_2, \overline{K}_2, P_2]$ and $C_3[\overline{K}_2, \overline{K}_2, \overline{K}_3]$. In particular, every 3-arc-strong semicomplete composition has a strong arc decomposition.
\end{thm}

A  \textbf{split digraph}  is a digraph whose vertex set is a disjoint union of two non-empty sets $V_1$ and $V_2$ such that $V_1$ is an independent set and the subdigraph induced by $V_2$ is semicomplete.  We use the notation $D=(V_1,V_2;A)$ to denote  a split digraph $D$.
 This class of digraphs has not been studied in many papers. In fact, our definition of a split digraph is different from that in \cite{hellDAM216,lamarDM312} but is very similar to the definition used in \cite{aiDM346}. Recall that an undirected graph is a split graph if its vertex set can be partitioned into an independent set and a clique. So by our definition, a split digraph is simply any digraph whose underlying undirected graph is a split graph. It should be noted that problems for split digraphs are often  much harder than the corresponding problem for semicomplete digraphs. One example is the hamiltonian cycle problem which is easy for semicomplete digraphs (a semicomplete digraph has a hamiltonian cycle if and only if it is strongly connected) but the complexity of the hamiltonian cycle problem is open already for split digraphs with only two vertices in the independent set. For split digraphs which become semicomplete after deletion of some vertex, the hamiltonian cycle problem is equivalent to deciding the existence of a hamiltonian path with prescribed starting and ending vertices and hence polynomial (but highly non-trivial) by \cite{bangJA13}.%See the discussion on page 290 in \cite{bang2009}.
%It also follows from the previous reference that for  general split digraphs the hamiltonian cycle problem is NP-complete.

It was shown in \cite{bangSJDM5} that it is NP-complete to decide for a given semicomplete digraph $D=(V,A)$ and a subset $A'\subset A$ whether $D$ has a hamiltonian cycle containing all arcs of $A'$. This immediately implies the following (for each arc $uv\in A'$ add a new vertex $x_{uv}$ and arcs $ux_{uv},x_{uv}v$).

\begin{thm}
 The hamiltonian cycle problem is NP-complete for split digraphs.
\end{thm}

As mentioned above the following conjecture is open already for $k=2$ (note that a digraph obtained as described below may not be a split digraph as we may add arcs between new vertices).

\begin{conj}\cite{bangCN115}
  For every fixed integer $k$ there exists a polynomial algorithm for deciding if there is a hamiltonian cycle in a given digraph which is obtained from a semicomplete digraph by adding $k$ new vertices and some arcs.
\end{conj}

Based on this conjecture and the discussion on page 290 in \cite{bang2009} we conjecture the following.

\begin{conj}
  For every fixed integer $k$ the hamiltonian cycle problem is polynomial for the class of split digraphs $D=(V_1,V_2;A)$ where the independent set $V_1$ has size at most $k$.
  \end{conj}

In this paper, we study strong arc decompositions of split digraphs. Our main result is the following.

\begin{thm}\label{thm-Arcdecsplit}
	Let $D=(V_1,V_2;A)$ be a 2-arc-strong split digraph such that $V_1$ is an independent set and the subdigraph induced by $V_2$ is semicomplete. If every vertex of $V_1$ has both out- and in-degree at least 3 in $D$, then $D$ has a strong arc decomposition.
\end{thm}

\begin{coro}\label{coro-3arcSAD}
	Every 3-arc-strong split digraph has a strong arc decomposition.
\end{coro}

The paper is organized as follows.  Section \ref{sec-prelim} provides additional terminology and notation and preliminary results. The proof of our main result, Theorem \ref{thm-Arcdecsplit}, is given in Section \ref{sec-proof}. In Section \ref{sec-GP}, we show some results on arc-disjoint in- and out-branchings in split digraphs and we provide an infinite family of split digraphs which shows that vertex-connectivity 2 is not sufficient to guarantee the existence of a strong arc decomposition in split digraphs (Corollary \ref{coro-2strNoSAD}). Finally we list some open problems in Section \ref{sec-problem}.

\section{Preliminaries}\label{sec-prelim}

Let $D$ be a directed multigraph and let $X$ be a subset of $V(D)$. We use $D\left\langle X\right\rangle$ to denote the directed multigraph induced by $X$. Let $D-X = D\left\langle V\backslash X \right\rangle$. We often identify a subdigraph $H$ of $D$ with its vertex set $V(H)$. For example, we write $D-H$ and $v\in H$ instead of $D-V(H)$ and $v\in V(H)$. 

In this paper a cycle and a path always means a directed cycle and path. For subsets $X,Y$ of $V(D)$, a path $P$ is an $(X,Y)$-path if it starts at a vertex $x$ in $X$ and ends at a vertex $y$ in $Y$ such that $V(P)\cap(X\cup Y)=\{x,y\}$. For a path $P$, we use $P[x,y]$ to denote the subpath of $P$ from $x$ to $y$. For every digraph $D$, we can label its strong components $C_1,\ldots{},C_p$ ($p\geq 1$) such that there is no arc from $C_j$ to $C_i$ when $j>i$. We call such an ordering an \textbf{acyclic ordering} of the strong components of $D$. For a semicomplete digraph $D$  it is easy to see that the ordering  $C_1,\ldots{},C_p$ ($p\geq 1$) is unique and we call $C_1$ (resp., $C_p$) the {\bf initial (resp., terminal)} strong component of $D$.

\subsection{Nice decompositions of strong digraphs}

A \textbf{vertex decomposition} of a digraph $D$ is a partition of its vertex set $V(D)$ into disjoint sets $(U_1,\ldots,U_l)$ ($l\geq 1$). The \textbf{index} of a vertex $v$ in the decomposition, denoted by $ind(v)$, is the index $i$ such that $v\in U_i$.  
An arc $xy$ is called a \textbf{backward arc} if $ind(x)>ind(y)$. The following decomposition, introduced in \cite{bangJGT102}, plays an important role in our proof of Theorem \ref{thm-Arcdecsplit}. A \textbf{nice decomposition} of a digraph $D$ is a vertex decomposition such that $D\left\langle U_i\right\rangle$ is strong for all $i\in[l]$ and the set of cut‐arcs of $D$ is exactly the set of backward arcs.% Thus the next observation follows from the defintion of a nice decomposition.

\begin{thm}\cite{bangJGT102}\label{thm-nicedecom}
	Every strong semicomplete digraph $D$ of order at least 4 admits a unique nice decomposition. Furthermore, the nice decomposition can be constructed in polynomial time.
\end{thm}

Given a semicomplete digraph and its nice decomposition, the \textbf{natural ordering} of its backward arcs is the ordering of these in decreasing order according to the index of their tails. 

\begin{prop}\cite{bangJGT102} \label{prop-nicedecom}
	Let $D$ be a strong semicomplete digraph of order at least 4. Suppose that $(U_1,\ldots,U_l)$ is the nice decomposition of $D$ and  $(x_1y_1,\ldots,x_ry_r)$ is the natural ordering of the backward arcs. Then the following statements hold.
	
	(i)  $x_1\in U_l$ and $y_r\in U_1$;
	
	(ii) $ind(y_{j+1})<ind(y_j)\leq ind(x_{j+1})< ind(x_j)$ for all $j\in[r-1]$ and $ind(y_{j+1})\leq ind(x_{j+2})<ind(y_j)$ for all $j\in[r-2]$.
\end{prop}

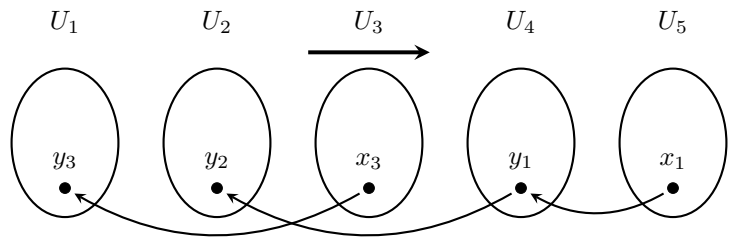
\begin{figure}[H]
	\centering\begin{tikzpicture}[scale=0.4]
	\foreach \i in {(-10,0),(-5,0),(0,0),(5,0),(10,0)}{\draw[ line width=0.8pt] \i ellipse [x radius=50pt, y radius=70pt];}
	\coordinate [label=center:$U_1$] () at (-10,4);
	\coordinate [label=center:$U_2$] () at (-5,4);
	\coordinate [label=center:$U_3$] () at (0,4);
	\coordinate [label=center:$U_4$] () at (5,4);
	\coordinate [label=center:$U_5$] () at (10,4);
	\draw[line width=1.5pt] (-2,3) -- (2,3); 

	\filldraw[black](-5,-1.5) circle (5pt)node[label=above:$y_2$](y2){};
	\filldraw[black](-10,-1.5) circle (5pt)node[label=above:$y_3$](y3){};
	\filldraw[black](10,-1.5) circle (5pt)node[label=above:$x_1$](x1){};
	\filldraw[black](5,-1.5) circle (5pt)node[label=above:$y_1$](y1){};
	\filldraw[black](0,-1.5) circle (5pt)node[label=above:$x_3$](x3){};
	
	\foreach \i/\j in {x1/y1,y1/y2,x3/y3}{\path[draw, line width=0.8pt] (\i) edge[bend left=30] (\j);}

	\end{tikzpicture}\caption{An illustration of a nice decomposition $(U_1,\ldots, U_5)$. The backward arcs are $\{x_1y_1,x_2y_2,x_3y_3\}$ with $x_2=y_1$. All arcs between different $U_i$s which are not shown are from left to right.}
	\label{fig-nicedecom}
\end{figure}

\subsection{Strong arc decompositions of semicomplete directed multigraphs}

Bang-Jensen, Gutin and Yeo generalized Theorem \ref{thm-SD} to  semicomplete directed multigraphs. The characterization below turns out to be extremely useful in our proof of Theorem \ref{thm-Arcdecsplit}.

\begin{thm}\cite{bangJGT95}\label{thm-SDM}
	A 2‐arc‐strong semicomplete directed multigraph $D = (V, A)$ on $n\geq 4$ vertices has a strong arc decomposition if and only if it is not isomorphic to one of the exceptional digraphs depicted in Figure \ref{fig-DMwithoutAD}. Furthermore, a strong arc decomposition of $D$ can be obtained in polynomial time when it exists.
\end{thm}

\begin{figure}[H]
	\centering
	\subfigure{\begin{minipage}[t]{0.23\linewidth}
			\centering\begin{tikzpicture}[scale=0.8]
				\filldraw[black](0,0) circle (3pt)node[label=left:$v_1$](v1){};
				\filldraw[black](2,0) circle (3pt)node[label=right:$v_2$](v2){};
				\filldraw[black](0,-2) circle (3pt)node[label=left:$v_3$](v3){};
				\filldraw[black](2,-2) circle (3pt)node[label=right:$v_4$](v4){};
				\foreach \i/\j/\t in {
					v1/v2/0,
					v2/v3/0,
					v3/v4/0,
					v4/v1/0,
					v1/v3/15,
					v2/v4/15,
					v3/v1/15,
					v4/v2/15
				}{\path[draw, line width=0.8] (\i) edge[bend left=\t] (\j);}			
			\end{tikzpicture}\caption*{$S_4$}\end{minipage}}
	\subfigure{\begin{minipage}[t]{0.23\linewidth}
			\centering\begin{tikzpicture}[scale=0.8]
				\filldraw[black](0,0) circle (3pt)node[label=left:$v_1$](v1){};
				\filldraw[black](2,0) circle (3pt)node[label=right:$v_2$](v2){};
				\filldraw[black](0,-2) circle (3pt)node[label=left:$v_3$](v3){};
				\filldraw[black](2,-2) circle (3pt)node[label=right:$v_4$](v4){};
				\foreach \i/\j/\t in {
					v1/v2/0,
					v2/v3/0,
					v3/v4/0,
					v4/v1/0,
					v1/v3/15,
					v2/v4/15,
					v3/v1/10,
					v3/v1/30,
					v4/v2/15
				}{\path[draw, line width=0.8] (\i) edge[bend left=\t] (\j);}	
			\end{tikzpicture}\caption*{$S_{4,1}$}\end{minipage}}
	\subfigure{\begin{minipage}[t]{0.23\linewidth}
			\centering\begin{tikzpicture}[scale=0.8]
				\filldraw[black](0,0) circle (3pt)node[label=left:$v_1$](v1){};
				\filldraw[black](2,0) circle (3pt)node[label=right:$v_2$](v2){};
				\filldraw[black](0,-2) circle (3pt)node[label=left:$v_3$](v3){};
				\filldraw[black](2,-2) circle (3pt)node[label=right:$v_4$](v4){};
				\foreach \i/\j/\t in {
					v1/v2/0,
					v1/v2/20,
					v2/v3/0,
					v3/v4/0,
					v4/v1/0,
					v1/v3/15,
					v2/v4/15,
					v3/v1/15,
					v4/v2/15
				}{\path[draw, line width=0.8] (\i) edge[bend left=\t] (\j);}	
			\end{tikzpicture}\caption*{$S_{4,2}$}\end{minipage}}
	\subfigure{\begin{minipage}[t]{0.23\linewidth}
			\centering\begin{tikzpicture}[scale=0.8]
				\filldraw[black](0,0) circle (3pt)node[label=left:$v_1$](v1){};
				\filldraw[black](2,0) circle (3pt)node[label=right:$v_2$](v2){};
				\filldraw[black](0,-2) circle (3pt)node[label=left:$v_3$](v3){};
				\filldraw[black](2,-2) circle (3pt)node[label=right:$v_4$](v4){};
				\foreach \i/\j/\t in {
					v1/v2/0,
					v2/v3/0,
					v3/v4/0,
					v4/v1/0,
					v1/v3/15,
					v2/v4/10,
					v2/v4/30,
					v3/v1/10,
					v3/v1/30,
					v4/v2/15
				}{\path[draw, line width=0.8] (\i) edge[bend left=\t] (\j);}	
			\end{tikzpicture}\caption*{$S_{4,3}$}\end{minipage}}
	\caption{2-arc-strong directed multigraphs without strong arc decompositions.}
	\label{fig-DMwithoutAD}
\end{figure}

\subsection{Splitting off arcs}

We now introduce a very useful operation that can be used to obtain various semicomplete directed multigraphs on the vertex set $V_2$ from a split digraph $D=(V_1,V_2;A)$. This will allow us to apply Theorem \ref{thm-SDM} and Lemma \ref{(D-X)->D} below in many cases to obtain the desired result for $D$. 

\begin{definition}\label{def-splitting}
  Suppose that $D=(V_1,V_2;A)$ is a split digraph and that $ut$ and $tv$ with $u\neq v$ are two arcs incident to a vertex $t\in V_1$. \textbf{Splitting off the pair $(ut, tv)$ at $t$} means that we replace the arcs $ut, tv$ by a new arc $uv$ (or a copy of that arc if it already exists). The arc $uv$ is called a \textbf{splitting arc}. The reverse operation where we replace a splitting arc by the two original arcs is called  \textbf{lifting }the arc.
\end{definition}

\begin{definition}
	 	 Suppose that $P$ is a path of a split digraph $D=(V_1,V_2;A)$ with both end-vertices in $V_2$. \textbf{Splitting off the path $P$} means that we split  off the pairs $(t^-t,tt^+)$ for every vertex $t\in V(P)\cap V_1$, where $t^-$ and $t^+$ are the predecessor and successor of $t$ on the path $P$, respectively.
\end{definition}

\begin{lem}\label{(D-X)->D}
	Let $D$ be a directed multigraph and let $X$ be a subset of $V(D)$ such that every vertex of $D-X$ has both two in-neighbors and two out-neighbors in $X$. If $X$ has a strong arc decomposition  then $D$ has a strong arc decomposition.
\end{lem}
\begin{proof}
	Let $(A_1,A_2)$ be a strong arc decomposition of $X$. By assumption, every vertex $x\in D-X$ has two out-neighbors $x_{1}^+, x_{2}^+$ and two in-neighbors $x_{1}^-,x_{2}^-$ in $X$. Then $A_i\cup\{x_{i}^-x, xx_{i}^+: x\in D-X\}$, $i\in[2]$ is a strong arc decomposition of $D$.
\end{proof}

\begin{lem}\label{lem-Dstar}
	Let $D=(V_1,V_2;A)$ be a split digraph such that every vertex of $V_1$ has out- and in-degree at least 3 in $D$. Let $D^{\ast}$ be a directed multigraph obtained from $D$ by splitting off at most two pairs at every vertex in $V_1$. If $D^{\ast}\left\langle V_{2}\right\rangle$ has a strong arc decomposition, then there is a strong arc decomposition of $D$. 
\end{lem}
\begin{proof}
	Let $A_1$ and $A_2$ be a strong arc decomposition of $D^{\ast}\left\langle  V_{2}\right\rangle$. Lifting every splitting arc $s_1s_2$, that is, replacing $s_1s_2$ with the two arcs $s_1t,ts_2$ in its corresponding splitting pair,  we obtain two disjoint arc sets $A_1^{\prime}, A_2^{\prime}$ of $A(D)$, see Figure \ref{fig-lem25}. %. In other words, for $i\in[2]$, the arc set $A_i^{\prime}$ is obtained from $A_i$ by deleting each splitting arc $s_1s_2\in A_i$ and adding the arcs $s_1t,ts_2$ in its splitting pair to $A_i$, see Figure \ref{fig-lem25}. 
	
	Let $D_i$ $(i\in[2])$ be the digraph induced by $A_i^{\prime}$. Clearly, each $D_i$ is strong  and it covers all vertices of $V_2$. Then every vertex of $V_1$ has at least 3 out- and in-neighbors in $V(D_i)$ for all $i\in[2]$. Recall that $D^{\ast}$ is obtained from $D$ by splitting off at most two pairs of every vertex in $V_1$. This implies that for each $i\in[2]$ and each vertex $t\in V(D_i)-V(D_{3-i})$, $t$ has at least one pair of in- and out-arcs which are not in $D_i$.  Adding such in- and out-arcs of $t$ in $A(D)-A(D_{i})$ to $A(D_{3-i})$, we obtain a strong arc decomposition of $D\left\langle V(D_1\cup D_2)\right\rangle$ and  the claim follows by Lemma \ref{(D-X)->D}.
\end{proof}

\begin{figure}[H]
	\centering
	\subfigure{\begin{minipage}[t]{0.3\linewidth}
			\centering\begin{tikzpicture}[scale=1]
				\coordinate [label=center:$V_2$] () at (2,2);
			\coordinate [label=center:$V_1$] () at (0,2);
			\filldraw[black](2,1) circle (3pt)node[label=right:$s_1$](s1){};
			\filldraw[black](2,-0.5) circle (3pt)node[](s2){};
			\filldraw[black](2,-2) circle (3pt)node[](s3){};
			\filldraw[black](2,-3.5) circle (3pt)node[label=right:$s_4$](s4){};
			\filldraw[black](0,-0) circle (3pt)node[label=left:$t$](t){};
			\filldraw[black](0,-2) circle (3pt)node[label=left:$t^{\prime}$](tt){};
			
			\foreach \i/\j/\t in {
				s1/s2/0,
				s3/s2/0,
				s3/s4/15,
				s4/s3/15,
				s2/s4/30,
				t/s1/9,
				s1/t/9,
				t/s2/9,
				s2/t/9,
				t/s3/9,
				s3/t/9,
				tt/s1/9,
				s1/tt/9,
				tt/s2/9,
				s2/tt/9,
				tt/s3/9,
				s3/tt/9
			}{\path[draw, line width=0.8] (\i) edge[bend left=\t] (\j);}
		\path[draw, line width=0.8pt] (s3) edge[bend right=30] (s1);
		\path[draw, line width=0.8pt] (s4) edge[bend right=50] (s1);
			\end{tikzpicture}\caption*{(a)}\end{minipage}}
	\subfigure{\begin{minipage}[t]{0.3\linewidth}
			\centering\begin{tikzpicture}[scale=1]
			\coordinate [label=center:$V_2$] () at (2,2);
			\coordinate [label=center:$V_1$] () at (0,2);
			\filldraw[black](2,1) circle (3pt)node[label=right:$s_1$](s1){};
			\filldraw[black](2,-0.5) circle (3pt)node[](s2){};
			\filldraw[black](2,-2) circle (3pt)node[](s3){};
			\filldraw[black](2,-3.5) circle (3pt)node[label=right:$s_4$](s4){};
			\filldraw[black](0,-0) circle (3pt)node[label=left:$t$](t){};
			\filldraw[black](0,-2) circle (3pt)node[label=left:$t^{\prime}$](tt){};
			
			\foreach \i/\j/\t/\w in {
				s1/s2/10/1.5,
				s3/s2/15/0.8,
				s4/s3/15/1.5,
				s2/s4/30/1.5,
				t/s1/9/0.8,
				s3/t/9/0.8,
				tt/s1/9/0.8,
				s1/tt/9/0.8,
				tt/s2/9/0.8,
				s2/tt/9/0.8,
				tt/s3/9/0.8,
				s3/tt/9/0.8
			}{\path[draw, line width=\w pt] (\i) edge[bend left=\t] (\j);}
			\path[draw, line width=1.5pt] (s3) edge[bend right=30] (s1);
			\path[draw, dashed, line width=0.8pt] (s4) edge[bend right=50] (s1);
			\path[draw, dashed, line width=0.8pt] (s1) edge[bend right=20] (s2);
			\path[draw, dashed, line width=0.8pt] (s2) edge[bend left=10] (s3);
			\path[draw, dashed, line width=0.8pt] (s3) edge[bend left=10] (s4);			\end{tikzpicture}\caption*{(b)}\end{minipage}}
		\subfigure{\begin{minipage}[t]{0.3\linewidth}
				\centering\begin{tikzpicture}[scale=1]
				\coordinate [label=center:$V_2$] () at (2,2);
				\coordinate [label=center:$V_1$] () at (0,2);
				\filldraw[black](2,1) circle (3pt)node[label=right:$s_1$](s1){};
				\filldraw[black](2,-0.5) circle (3pt)node[](s2){};
				\filldraw[black](2,-2) circle (3pt)node[](s3){};
				\filldraw[black](2,-3.5) circle (3pt)node[label=right:$s_4$](s4){};
				\filldraw[black](0,-0) circle (3pt)node[label=left:$t$](t){};
				\filldraw[black](0,-2) circle (3pt)node[label=left:$t^{\prime}$](tt){};
				
				\foreach \i/\j/\t/\w in {
					s1/s2/0/1.5,
					s3/s2/0/0.8,
					s4/s3/15/1.5,
					s2/s4/30/1.5,
					t/s1/9/1.5,
					s3/t/9/1.5,
					tt/s1/9/0.8,
					s1/tt/9/0.8,
					tt/s2/9/0.8,
					s2/tt/9/0.8,
					tt/s3/9/0.8,
					s3/tt/9/0.8
				}{\path[draw, line width=\w pt] (\i) edge[bend left=\t] (\j);}
				\path[draw, line width=1.5pt] (s3) edge[bend right=30] (s1);
				\foreach \i/\j in {s1/t,t/s2,s2/t,t/s3}{\path[draw, dashed, line width=0.8pt] (\i) edge[bend left=9] (\j);}
				\path[draw, dashed, line width=0.8pt] (s3) edge[bend left=10] (s4);	
					\path[draw, dashed, line width=0.8pt] (s4) edge[bend right=50] (s1);
				\end{tikzpicture}\caption*{(c)}\end{minipage}}
	\caption{The split digraph $D=(V_1,V_2;A)$ with $V_1=\{t,t^{\prime}\}$ and $V_2=\{s_1,s_2,s_3,s_4\}$ such that there is a 2-cycle between every vertex in $V_1$ and every vertex in $\{s_1,s_2,s_3\}$, see (a). The vertices of $V_2$ in (a)-(c) are labeled from top to bottom. We split off $(s_1t,ts_2)$ and $(s_2t,ts_3)$ at $t$ to obtain $D^{\ast}$ and then $D^{\ast}\left\langle  V_{2}\right\rangle=D\left\langle  V_2\right\rangle\cup\{s_1s_2,s_2s_3\}$, see (b). Note that the dashed arc set $A_1$ and bold arc set $A_2$ form a strong arc decomposition of $D^{\ast}\left\langle  V_2\right\rangle$. Lifting the arcs $s_1s_2, s_2s_3$ to get $A_1^{\prime},A_2^{\prime}$ and adding a pair of arcs $s_3t,ts_1$ to $A_2^{\prime}$, we obtain a strong arc decomposition of $D\left\langle  V_2\cup\{t\}\right\rangle$,  which is shown by dashed and bold arcs in (c). It should be noted that $t^{\prime}$ is not covered by $A_1^{\prime}\cup A_2^{\prime}$, so \yw{by Lemma \ref{(D-X)->D},} we do not need to consider it.}
	\label{fig-lem25}
\end{figure}

%\begin{remark}\label{Remark-notallatmost2}
%	By the argument in the proof of Lemma \ref{lem-Dstar},  the lemma also holds if replace the condition `splitting off at most two pairs at every vertex in $V_1$' by `splitting off at most two pairs at every vertex in $V_1$ which not covered by $A_1^{\prime}\cap A_2^{\prime}$', where  $A_1^{\prime}$ and $A_2^{\prime}$ are obtained from a strong arc decomposition $A_1,A_2$ of $D^{\ast}\left\langle V_{2}\right\rangle$ by lifting all splitting arcs.
%\end{remark}

The next lemma covers  the cases where the semicomplete part of the split digraph has size at most 3 or it is one of the exceptions in Theorem \ref{thm-SDM}

\begin{lem}\label{lem-DV2small}
	Let $D=(V_1,V_2;A)$ be a split digraph such that every vertex of $V_1$ has out- and in-degree at least 3 in $D$. Then $D$ has a strong arc decomposition if one of the following statements holds.
	
	(i) $D$ is 2-arc-strong and $|V_2|\leq 3$.
	
	(ii) $D\left\langle  V_{2}\right\rangle$ is isomorphic to one of the digraphs shown in Figure \ref{fig-DMwithoutAD}.
\end{lem}
\begin{proof}
	Let $t$ be an arbitrary vertex of $V_1$. Observe that $|V_2|\geq 3$ as $t$ has in- and out-degree at least three in $D$. Further, if $|V_2|=3$, then there is a 2-cycle between $t$ and each vertex of $V_2$. 
	
	First we consider the case $|V_2|=3$. If $D\left\langle  V_{2}\right\rangle$ contains a 3-cycle, say $s_1s_2s_3s_1$, then the arcs of $ts_1s_2s_3t$ and $ts_3s_1t\cup ts_2t$ form a strong arc decomposition of  $D\left\langle  V_{2}\cup\{t\}\right\rangle$ and now the claim follows by Lemma \ref{(D-X)->D}. So we may assume that $D\left\langle  V_{2}\right\rangle$ contains no 3-cycle. Then there is a vertex of $V_2$ which has in-degree or out-degree zero in $D\left\langle  V_{2}\right\rangle$.  Since $D$ is 2-arc-strong, $V_1$ has size at least two. Let $t^{\prime}$ be a vertex of $V_1-\{t\}$ and let $s_1s_2s_3$ be a hamiltonian path of $D\left\langle  V_{2}\right\rangle$. Then the arcs of $ts_1s_2t^{\prime}s_3t$ and $ts_2s_3t^{\prime}s_1t$ form a strong arc decomposition of $D\left\langle  V_{2}\cup\{t,t^{\prime}\}\right\rangle$. Thus the claim follows by Lemma \ref{(D-X)->D} again.

To prove (ii) it suffices to show the case $D\left\langle  V_{2}\right\rangle=S_4$ as each digraph in Figure \ref{fig-DMwithoutAD} contains $S_4$ has a spanning subdigraph.  Since $t$ has in- and out-degree at least three in $D$, the vertex $t$ has an in-neighbor and an out-neighbor in $\{v_1,v_3\}$ (resp., in $\{v_2,v_4\}$). Let $A_1$ be the union of the four arcs between $t$ and such neighbors and arcs in $\{v_1v_3,v_3v_1,v_2v_4,v_4v_2\}$. As $t$ has in- and out-degree at least 3, there exist $t^-,t^+\in V_2$ such that $t^-t,tt^+$ is a pair of in- and out-arcs of $t$  in $A(D)-A_1$.  Let $A_2=\{t^-t,tt^+,v_1v_2,v_2v_3,v_3v_4,v_4v_1\}$. Then $A_1$ and $A_2$ form a strong arc decomposition of $D\left\langle  V_{2}\cup\{t\}\right\rangle$ and now the claim follows by Lemma \ref{(D-X)->D}.
\end{proof}

\begin{lem}\label{lem-D2starsamll}
	Let $D=(V_1,V_2;A)$ be a split digraph such that every vertex of $V_1$ has out- and in-degree at least 3 in $D$. Let $D^{\ast}$ be any directed multigraph obtained from $D$ by splitting off some pairs at every vertex in $V_1$. If $D^{\ast}\left\langle  V_{2}\right\rangle$ is isomorphic to one of the digraphs in Figure \ref{fig-DMwithoutAD}, then $D$ has a strong arc decomposition.
\end{lem}

\begin{proof}
	 Since $D\left\langle  V_{2}\right\rangle$ is a semicomplete digraph, it contains at least one arc between every pair of distinct vertices of $V_2$. So we may assume that the arcs in $A_1=\{v_1v_2,v_2v_3,v_3v_4,v_4v_1\}$ are original arcs, in other words, they belong to $D\left\langle  V_{2}\right\rangle$.  By Lemma \ref{lem-DV2small} (ii), we may further assume that $D^{\ast}\left\langle  V_{2}\right\rangle$ contains at least one splitting arc.  Observe that as $D\left\langle  V_{2}\right\rangle$ has no parallel arcs, at least one arc from each pair of parallel arcs must be a splitting arc in  $D^{\ast}\left\langle  V_{2}\right\rangle$. So the arc from $v_1$ to $v_2$ in $A(S_{4,2})-A_1$ is a splitting arc. 
	 
	Suppose first that there is no splitting arc of $D^{\ast}\left\langle  V_{2}\right\rangle$ with both end-vertices in $\{v_1,v_3\}$ or in $\{v_2,v_4\}$. Then $D^{\ast}\left\langle  V_{2}\right\rangle$ is isomorphic to the digraph $S_{4,2}$ and the arc from $v_1$ to $v_2$ in $A(S_{4,2})-A_1$ is the only splitting arc. Let $(v_1t,tv_2)$ be its corresponding splitting pair and let $t^-t,tt^+$ be two arcs in $D$ with $t^-\in\{v_2,v_4\}$ and $t^+\in\{v_1,v_3\}$. The arcs $t^-t$ and $tt^+$ exist as $t$ has out-degree and in-degree at least three in $D$. Moreover, there exists a pair of in- and out-arcs $e^-,e^+$ of $t$ in $D$ which are distinct from the arcs in $\{v_1t,tv_2,t^-t,tt^+\}$. Then $\{e^+,e^-\}\cup A_1$ and $\{v_1t,tv_2,t^-t,tt^+,v_1v_3,v_3v_1,v_2v_4,v_4v_2\}$ is a strong arc decomposition of $D\left\langle  V_{2}\cup\{t\}\right\rangle$ and the claim follows by Lemma \ref{(D-X)->D}.

        \iffalse
	 \begin{figure}[!h]
		\centering
		\includegraphics[scale=0.15]{lem26.jpg}
		\caption{The case that $C_1$ and $C_2$ are vertex disjoint in Lemma \ref{lem-D2starsamll}.}
		\label{fig-lem26}
	 \end{figure}
         \fi

	If there is a splitting arc of $D^{\ast}\left\langle  V_{2}\right\rangle$ with both end-vertices in $\{v_1,v_3\}$ (resp., in $\{v_2,v_4\}$), then let $ab$ with $\{a,b\}= \{v_1,v_3\}$ (resp., $cd$ with $\{c,d\}=\{v_2,v_4\}$) be such an arc. We may further assume that  $(at_1,t_1b)$  and $(ct_2,t_2d)$ are the corresponding splitting pairs of $ab$ and $cd$, respectively, if $ab$ or $cd$ exists.   Clearly, $C_1=aba$ or $C_1=at_1ba$ (resp., $C_2=cdc$ or $C_2=ct_2dc$) is a cycle of the original split digraph $D$. Observe that $C_1$ and $C_2$ are vertex disjoint or they share exactly one common vertex.  For the case that $C_1$ and $C_2$ share one vertex, both $ab$ and $cd$ exist and $t_1=t_2$. Let $t_1^-t_1,t_1t_1^+$ be a pair of arcs in $A(D)-A(C_1\cup C_2)$. Then $A_1\cup\{t_1^-t_1,t_1t_1^+\}$ and $A(C_1\cup C_2)$ is a strong arc decomposition of $D\left\langle  V_{2}\cup\{t_1\}\right\rangle$ and the claim follows by Lemma \ref{(D-X)->D}. 
	
	It remains to consider the case where $C_1$ and $C_2$ are vertex disjoint. % (see Figure \ref{fig-lem26}). 
	If $ab$ exits, let $t_1^-t_1, t_1t_1^+$ be two arcs with $t_1^-,t_1^+$ in $\{v_2,v_4\}$ and, if $cd$ exits, let $t_2^-t_2, t_2t_2^+$ be two arcs with $t_2^-,t_2^+$ in $\{v_1,v_3\}$.  The arcs $t_i^-t_i, t_it_i^+$ $i\in[2]$ exist as every vertex of $V_1$ has out- and in-degree at least three.  Since one of $ab$ and $cd$ is a splitting arc, one of $t_1$ and $t_2$, say $t_j$, exists. Let $A_2=A(C_1\cup C_2)\cup\{t_j^-t_j,t_jt_j^+\}$ and let $e^+,e^-$ be an out- and in-arc of $t_j$ in $A(D)-A_2$. Then $A_1\cup \{e^+,e^-\}$ (or $A_1\cup \{e^+,e^-,t_{3-j}^-t_{3-j},t_{3-j}t_{3-j}^+\}$ if $t_{3-j}$ exists) and $A_2$ is a strong arc decomposition of $D\left\langle  V_{2}\cup\{t_1\}\cup\{t_2\}\right\rangle$. By Lemma \ref{(D-X)->D}, $D$ has a strong arc decomposition.
\end{proof}

Combining  Theorem \ref{thm-SDM} and Lemmas \ref{lem-Dstar}, \ref{lem-D2starsamll}, we obtain the following result which indicates that what we seek is a number of arc-disjoint paths such that when we split off these, the semicomplete part of the resulting directed multigraph is 2-arc-strong.

\begin{coro}\label{coro-DV2star2ArcStr}
	Let $D=(V_1,V_2;A)$ be a split digraph such that every vertex of $V_1$ has out- and in-degree at least 3 in $D$. Let $D^{\ast}$ be a directed multigraph obtained from $D$ by splitting off at most two pairs of every vertex in $V_1$. If $D^{\ast}\left\langle  V_{2}\right\rangle$ is 2-arc-strong, then $D$ has a strong arc decomposition.
\end{coro}

\section{The proof of Theorem \ref{thm-Arcdecsplit}}\label{sec-proof}

For convenience we repeat the statement of the theorem.\\

\noindent {\bf Theorem \ref{thm-Arcdecsplit}}
Let $D=(V_1,V_2;A)$ be a 2-arc-strong split digraph such that $V_1$ is an independent set and the subdigraph induced by $V_2$ is semicomplete. If every vertex of $V_1$ has both out- and in-degree at least 3 in $D$, then $D$ has a strong arc decomposition. 

\bigskip
 We first have the following simple observations, which the first follows immediately by considering whether the vertex $x$ belongs to $S_k$ or not, and the second one can be checked easily as $xy$ must be a backward arc of $S$.
 
 \begin{obse}\label{obse-smallvertex}
 	Suppose that $xy$ is the only out-arc of some vertex $x$ in $D\left\langle  V_{2}\right\rangle$. Then one of the following holds.
 	
 	(i) $V_{S_{k-1}}=\{x\}$ and $V(S_k)=\{y\}$.
 	
 	(ii) $x,y\in V(S_k)$ and $xy$ is a cut-arc of $S_k$.
 \end{obse}
 
 Recall that the backward arcs of a nice decomposition are exactly the cut-arcs.
 \begin{obse}\label{obse-smallvtx-in-gd}
 	Let $S$ be a strong semicomplete digraph with $|V(S)|\geq 4$ and let $(U_1,\ldots, U_l)$ be the nice decomposition of $S$ and $(x_1y_1,\ldots, x_ry_r)$  the natural ordering of its backward arcs. If $xy\in A(S)$ is the only out-arc of $x$ in $S$, then one of the following holds.
 	
 	(i) $U_l=\{x\}=\{x_1\}$ and $y=y_1$.
 	
 	(ii) $U_l=\{x_1\}$, $U_{l-1}=\{x_2\}=\{y_1\}=\{x\}$ and $y=y_2$.
 \end{obse}

By Lemma \ref{lem-DV2small} (i), we may assume that $|V_2|\geq 4$ and by Corollary \ref{coro-DV2star2ArcStr}, we may assume that $D\left\langle  V_{2}\right\rangle$ is not 2-arc-strong. Let $S_1,\ldots, S_k$ $(k\geq 1)$ be the acyclic ordering of the strong components of $D\left\langle  V_{2}\right\rangle$. If $|V(S_k)|\geq 4$, then by Theorem \ref{thm-nicedecom}, $V(S_k)$ has a nice decomposition. Let $(U_{1},\ldots, U_{l})$ be its nice decomposition and let $(x_1y_1,\ldots, x_ry_r)$ be the natural ordering of its backward arcs.  Similarly, suppose that  $(W_{1},\ldots, W_{p})$ is the nice decomposition of $V(S_1)$ when $|V(S_1)|\geq 4$.\\

\noindent{}Let the subsets $X,Y$ of $V_2$ be defined as follows. Note that $X\neq Y$ since $D\left\langle  V_{2}\right\rangle$ is either non-strong or it is strong and has at least one cut-arc.

$$X=\begin{cases}
U_{l}, &\mbox{if } |V(S_k)|\geq 4;\\
V(S_k), &\mbox{otherwise}.
\end{cases} \mbox{ \;\;and \;\;} Y=\begin{cases}
W_{1}, &\mbox{if } |V(S_1)|\geq 4;\\
V(S_1), &\mbox{otherwise}.
\end{cases}$$

	\begin{lem}\label{lem-DPstrong}
		For each $(X,Y)$-path $P$ in $D$, the directed multigraph $D_{P}\left\langle  V_2\right\rangle$ obtained from $D\left\langle  V_2\right\rangle$ by splitting off $P$ is strong. 
	\end{lem}
\begin{proof}
	Observe that $D_{P}\left\langle  V_2\right\rangle$ can be obtained from $D\left\langle  V_2\right\rangle$ by adding some splitting arcs.  So we may assume that $D\left\langle  V_2\right\rangle$ is non-strong since otherwise there is nothing to prove. 
	
	Suppose that $P$ starts from $a$ and ends at $b$. Clearly, $a\in X\subseteq V(S_k)$ and $b\in Y\subseteq V(S_1)$. Further, for any vertex $x\in V_2$, it can reach $a$ by the arc $xa$ \yw{if $x\notin V(S_k)$} or an $(x,a)$-path in $S_k$, similarly, any vertex $y\in V_2$ can reached from  $b$. Then there exists an $(x,y)$-path consisting of such $(x,a)$-, $(b,y)$-paths and \yw{the} path $P^{\prime}$, where $P^{\prime}$ is the path in $D_{P}\left\langle  V_2\right\rangle$ which is obtained by splitting off $P$. In other words, for any vertices $x,y$, there is an $(x,y)$-path in $D_{P}\left\langle  V_2\right\rangle$ and \yw{thus} $D_{P}\left\langle  V_2\right\rangle$ is strong.
\end{proof}

Suppose that $\mathcal{Q}=(Q_1,\ldots,Q_{\alpha})$ ($\alpha\geq 2$) is an ordering of arc-disjoint paths such that $Q_1$ and $Q_2$ are $(X,Y)$-paths and $Q_i=u_it_iv_i$ ($i\geq 3$) with $t_i\in V_1$. For convenience, we also regard $\mathcal{Q}$ as a set of those paths. The set $\mathcal{Q}$ is called \textbf{$\gamma$-feasible} if every vertex of $V_1$ belongs to at most $\gamma$ paths of $\mathcal{Q}$. \textbf{Splitting off $\mathcal{Q}$} means that split off all paths in $\mathcal{Q}$. We often use the symbol $D_{\mathcal{Q}}$ to denote the directed multigraph obtained from $D$ by splitting off all paths in $\mathcal{Q}$.

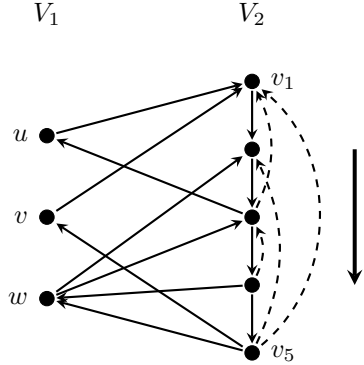
\begin{figure}[H]
	\centering\begin{tikzpicture}[scale=0.9]
	\coordinate [label=center:$V_2$] () at (2,2);
	\coordinate [label=center:$V_1$] () at (-1,2);
	\draw[line width=1.5pt] (3.5,0) -- (3.5,-2); 
	
	\filldraw[black](2,1) circle (3pt)node[label=right:$v_1$](v1){};
	\filldraw[black](2,0) circle (3pt)node[](v2){};
	\filldraw[black](2,-1) circle (3pt)node[](v3){};
	\filldraw[black](2,-2) circle (3pt)node[](v4){};
	\filldraw[black](2,-3) circle (3pt)node[label=right:$v_5$](v5){};

\filldraw[black](-1,0.2) circle (3pt)node[label=left:$u$](u){};
\filldraw[black](-1,-1) circle (3pt)node[label=left:$v$](v){};
\filldraw[black](-1,-2.2) circle (3pt)node[label=left:$w$](w){};

	\foreach \i/\j in {v1/v2,v2/v3,v3/v4,v4/v5,v5/w,v4/w,w/v3,w/v2,v5/v,v/v1,v3/u,u/v1}{\path[draw, line width=0.8pt] (\i) edge[] (\j);}
	\foreach \i/\j in {v5/v2,v3/v1,v4/v3}{\path[draw, dashed, line width=0.8pt] (\i) edge[bend right=25] (\j);}
	\path[draw, dashed, line width=0.8pt] (v5) edge[bend right=50] (v1);
	\end{tikzpicture}
	\caption{An illustration of the splitting off operation on a 2-feasible set $\mathcal{Q}$. $D\left\langle  V_2\right\rangle$ is a transitive tournament of order 5 with vertices $v_1,\ldots,v_5$ and all arcs from top to bottom.  So $X=V(S_k)=\{v_5\}$  and $Y=\{v_1\}$. The 2-feasible set $\mathcal{Q}$ consists of three paths $Q_1,Q_2$ and $Q_3$, where $Q_1=v_5wv_2v_3uv_1$, $Q_2=v_5vv_1$ and $Q_3=v_4wv_3$. The dashed arcs are the splitting arcs obtained by splitting off $\mathcal{Q}$. Note that $D_{\mathcal{Q}}\left\langle  V_2\right\rangle$ can be constructed from $D\left\langle  V_2\right\rangle$ by adding all dashed arcs.}
	\label{fig-fesibleset}
\end{figure}

It follows by Lemma \ref{lem-DPstrong} that every vertex of $V_2$ has both out- and in-degree at least one in $D_{\mathcal{Q}}\left\langle  V_2\right\rangle$. Let $W_{\mathcal{Q}}^+\subseteq V_2$ (resp., $W_{\mathcal{Q}}^-\subseteq V_2$) denote the set of vertices with out- (resp., in-)degree exactly one in $D_{\mathcal{Q}}\left\langle  V_2\right\rangle$. Next we show that if $W_{\mathcal{Q}}^+=W_{\mathcal{Q}}^-=\emptyset$, then either $D_{\mathcal{Q}}\left\langle  V_{2}\right\rangle$ is 2-arc-strong or it contains at most two cut-arcs.

\begin{lem}\label{lem-cutarc}
	Let $Q=(Q_1,\ldots,Q_{\alpha})$ ($\alpha\geq 2$) be a $\gamma$-feasible set. Suppose that $W_{\mathcal{Q}}^+=W_{\mathcal{Q}}^-=\emptyset$. If $D_{\mathcal{Q}}\left\langle  V_{2}\right\rangle$ is not 2-arc-strong, then every cut-arc $xy$ \yw{of $D_{\mathcal{Q}}\left\langle  V_{2}\right\rangle$} satisfies one of the following statements.
	
	(i) $xy\in \yw{A(S_k)}$, $|V(S_k)|=3$ and the two $(X,Y)$-paths $Q_1,Q_2$ in $\mathcal{Q}$ are starting from the vertex $y$;
	
	(ii) $xy\in \yw{A(S_1)}$, $|V(S_1)|=3$ and the two $(X,Y)$-paths $Q_1,Q_2$  in $\mathcal{Q}$ are ending at the the vertex $x$;
	
	Moreover, there exists a $\gamma$-feasible set $\mathcal{Q^{\prime}}$ such that $D_{\mathcal{Q^{\prime}}}\left\langle  V_{2}\right\rangle$ is  2-arc-strong.
\end{lem}
\begin{proof}
	By Lemma \ref{lem-DPstrong}, we may assume that $D_{\mathcal{Q}}\left\langle  V_{2}\right\rangle$ is 1-arc-strong with a cut-arc $xy$. First we claim that $xy$ is an original arc of $D\left\langle  V_{2}\right\rangle$. Suppose not, say $xy$ is a splitting arc of $D_{\mathcal{Q}}\left\langle  V_{2}\right\rangle$ and it is obtained by splitting off the path $Q_i\in\mathcal{Q}$. By the fact that $\mathcal{Q}-Q_i$ has an $(X,Y)$-path and Lemma \ref{lem-DPstrong} again, the directed multigraph $D_{\mathcal{Q}-Q_i}\left\langle  V_{2}\right\rangle$ obtained from $D\left\langle  V_{2}\right\rangle$ by splitting off all paths in $\mathcal{Q}-Q_i$ is also strong. This contradicts the fact that $xy$ is a cut-arc of $D_{\mathcal{Q}}\left\langle  V_{2}\right\rangle$ as $A(D_{\mathcal{Q}-Q_i}\left\langle  V_{2}\right\rangle)\subseteq A(D_{\mathcal{Q}}\left\langle  V_{2}\right\rangle-xy)$.
	
	\yw{Note that at most one of the paths $Q_1, Q_2$ contains the arc $xy$ so suppose w.l.o.g that $Q_1$ avoids $xy$.} Let $P$ be an arbitrary $(X,Y)$-path in $D_{\mathcal{Q}}\left\langle  V_{2}\right\rangle-xy$. Such \yw{a} path exists as it can be obtained by splitting off the  $(X,Y)$-path $Q_1$ in $D$. Suppose that $P$ starts from $a$ and ends at $b$. Clearly $a\in X\subseteq V(S_k)$ and $b\in Y\subseteq V(S_1)$.
	
	As $xy$ is a cut-arc of $D_{\mathcal{Q}}\left\langle  V_{2}\right\rangle$, there is no $(x,a)$-path or no $(b,y)$-path in $D_{\mathcal{Q}}\left\langle  V_{2}\right\rangle-xy$, consequently, \yw{the same holds} in $D\left\langle  V_{2}\right\rangle-xy$. W.l.o.g we may assume that  $D\left\langle  V_{2}\right\rangle-xy$ contains no $(x,a)$-path. \yw{Suppose first that} $x\notin V(S_k)$. Then $x$ dominates $a$. So we have $V(S_k)=\{a\}=\{y\}$ and $V(S_{k-1})=\{x\}$. Since there is no $(x,a)$-path in $D_{\mathcal{Q}}\left\langle  V_{2}\right\rangle-xy$, $xy$ is also the only out-arc of $x$ in $D_{\mathcal{Q}}\left\langle  V_{2}\right\rangle$, which implies that $x\in W_{\mathcal{Q}}^+$ and this contradicts our assumption. Therefore, we may assume that $x\in V(S_k)$. Then $y\in V(S_k)$ as $xy\in A(D\left\langle  V_{2}\right\rangle)$. Moreover, there is no $(x,a)$-path in $S_k-xy$, which means that $xy$ is also a cut-arc of $S_k$.

	If $|V(S_k)|\geq 4$, then $a\in U_l$ and $xy=x_iy_i$ for some $i\in[r]$. Since there is no $(x_i,a)$-path in $S_k-x_iy_i$, we have $x_i\neq x_1$ and $x_i$ does not dominate $a$. Note that there is only one arc, i.e., $x_1y_1$, leaving the set $U_l$ in $D\left\langle  V_{2}\right\rangle$ \yw{and every vertex in $U_1\cup\cdots\cup U_{l-1}$ except $y_1$ dominates all vertices of $U_l$}. So $a=x_1, x=x_i=y_1$ and $x_iy_i$ is the only out-arc of $x$ in $D\left\langle  V_{2}\right\rangle$. This means that $U_l=\{x_1\}=\{a\}$,  $U_{l-1}=\{x\}$ and $xy=x_2y_2$ by Proposition \ref{prop-nicedecom} (ii). Moreover, since $x$ can not reach $a$ in $D_{\mathcal{Q}}\left\langle  V_{2}\right\rangle-xy$, the arc $xy$ is also the only out-arc of $x$ in  $D_{\mathcal{Q}}\left\langle  V_{2}\right\rangle$. Then $x\in W_{\mathcal{Q}}^+$, which contradicts our assumption. 
	
	For the case $|V(S_k)|=2$, we have $S_k=xyx$ and $y=a$. Moreover, the arc $xy$ is also the only out-arc of $x$ in  $D_{\mathcal{Q}}\left\langle  V_{2}\right\rangle$ and then $x\in W_{\mathcal{Q}}^+$, a contradiction again.
	
	So it suffices to consider the case $|V(S_k)|=3$. As $S_k$ is strong \yw{and $xy$ is a cut-arc}, we may assume that $xyzx$ is a 3-cycle in $S_k$. Recall that there is no $(x,a)$-path in $S_k-xy$ and in $D_{\mathcal{Q}}\left\langle  V_{2}\right\rangle-xy$. Clearly $x\neq a$. Further, $x$ can not reach any vertex $u$ of $V_2-V(S_k)$ in $D_{\mathcal{Q}}\left\langle  V_{2}\right\rangle-xy$ as $u$ dominates $a$.  Since $x\notin  W_{\mathcal{Q}}^+$, either $xz\in A(S_k)$ or there exists a path $Q_i\in\mathcal{Q}$ with $i\geq 3$ such that $Q_i=xtz$. In both cases, $x$ can reach $z$ in $D_{\mathcal{Q}}\left\langle  V_{2}\right\rangle-xy$, which implies that $z\neq a$ and then $y=a$. In the same way, $z$ can not reach $a$ and \yw{can not reach any} vertex of $V_2-V(S_k)$ in $D_{\mathcal{Q}}\left\langle  V_{2}\right\rangle-xy$. As $z\notin  W_{\mathcal{Q}}^+$,  there exists a path $Q_j\in\mathcal{Q}$ with $i\neq j$ and $i\geq 3$ such that $Q_j=zt^{\prime}x$,  see Figure \ref{fig-lem31} (a). Since $a$ is the first vertex of an arbitrary $(X,Y)$-path in $D_{\mathcal{Q}}\left\langle  V_{2}\right\rangle-xy$ and $xy\in A(S_k)$, we have that both $Q_1$ and $Q_2$ are starting from \yw{$a=y$}, which implies (i). By symmetry, if there is no $(b,y)$-path in $D\left\langle  V_{2}\right\rangle-xy$, then (ii) holds.
		
		Next we show that there exists a $\gamma$-feasible set $\mathcal{Q^{\prime}}$ such that $D_{\mathcal{Q^{\prime}}}\left\langle  V_{2}\right\rangle$ is 2-arc-strong. Since $t^{\prime}$ has out-degree at least 3, it has an out-neighbor $w$ in $V_2-\{x,z\}$. Let $\mathcal{Q^{\prime}}=(\mathcal{Q}-zt^{\prime}x)\cup\{zt^{\prime}w\}$. It can be checked easily that $\mathcal{Q^{\prime}}$ is also a $\gamma$-feasible set and every vertex of $V_2$ has the same out-degree (resp., every vertex of $V_2-x$ has the same in-degree) in $D_{\mathcal{Q^{\prime}}}\left\langle  V_{2}\right\rangle$ and in $D_{\mathcal{Q}}\left\langle  V_{2}\right\rangle$. \yw{Furthermore},  $d_{D_{\mathcal{Q^{\prime}}}\left\langle  V_{2}\right\rangle}^-(x)\geq d_{D\left\langle  V_{2}\right\rangle}^-(x)\geq 2$ as $|V_2|\geq 4$. Thus $W_{\mathcal{Q^{\prime}}}^+=W_{\mathcal{Q^{\prime}}}^-=\emptyset$. Moreover, for each arc $uv\in V(S_k)$, there is a $(u,v)$-path in $D_{\mathcal{Q^{\prime}}}\left\langle  V_{2}\right\rangle-uv$. In other words, \yw{no} arc of $S_k$ is a cut-arc of $D_{\mathcal{Q^{\prime}}}\left\langle  V_{2}\right\rangle$. Similarly, we can adjust $Q^{\prime}$ a little bit to obtain a new $\gamma$-feasible set, still call it $Q^{\prime}$, such that $W_{\mathcal{Q^{\prime}}}^+=W_{\mathcal{Q^{\prime}}}^-=\emptyset$ and \yw{no} arc of $S_1$ and $S_k$ is a cut-arc of $D_{\mathcal{Q^{\prime}}}\left\langle  V_{2}\right\rangle$. Then statements (i) and (ii) show that  $D_{\mathcal{Q^{\prime}}}\left\langle  V_{2}\right\rangle$ is  2-arc-strong and this completes the proof. 
\end{proof}
\begin{figure}[H]
	\centering
	\subfigure{\begin{minipage}[t]{0.33\linewidth}
			\centering\begin{tikzpicture}[scale=0.7]	
			\coordinate [label=center:$V_2$] () at (3,1.8);
			\coordinate [label=center:$V_1$] () at (0,1.8);
			
			\filldraw[black](3,1) circle (3pt)node[label=right:$b$](b){};
			\filldraw[black](3,-1.5) circle (3pt)node[label=right:$y$](y){};
			\filldraw[black](4,-4) circle (3pt)node[label=right:{$x$}](x){};
			\filldraw[black](2,-4) circle (3pt)node[label=left:$z$](z){};			
			
			\filldraw[black](0,-2.5) circle (3pt)node[label=left:$t^{\prime}$](tt){};
			\filldraw[black](0,-1.5) circle (3pt)node[label=left:$v$](v){};
			\filldraw[black](0,-0.5) circle (3pt)node[label=left:$u$](u){};
			\foreach \i/\j/\b in {
				x/y/0,
				y/z/0,
				z/x/10,
				x/z/10,
				z/tt/0,
				tt/y/0,
				y/v/0,
				y/u/0,
				u/b/0,
				v/b/0,
				b/y/0,
				b/z/10
			}{\path[draw, line width=0.8pt] (\i) edge[bend right=\b] (\j);}
		\path[draw, line width=0.8pt] (tt) edge[bend left=10] (x);
			\path[draw, line width=0.8pt] (b) edge[bend left=10] (x);
			\end{tikzpicture}\caption*{(a) $D$}\end{minipage}}
		\subfigure{\begin{minipage}[t]{0.3\linewidth}
				\centering\begin{tikzpicture}[scale=0.7]	
				\filldraw[black](3,1) circle (3pt)node[label=right:$b$](b){};
				\filldraw[black](3,-1.5) circle (3pt)node[label=right:$y$](y){};
				\filldraw[black](4,-4) circle (3pt)node[label=right:{$x$}](x){};
				\filldraw[black](2,-4) circle (3pt)node[label=left:$z$](z){};			
				
				\foreach \i/\j/\b in {
					x/y/0,
					y/z/0,
					z/x/10,
					x/z/10,
					b/y/0,
					b/z/10
				}{\path[draw, line width=0.8pt] (\i) edge[bend right=\b] (\j);}
				\path[draw, line width=0.8pt] (b) edge[bend left=10] (x);
				\path[draw, dashed, line width=0.8pt] (y) edge[bend left=50] (b);
				\path[draw, dashed, line width=0.8pt] (y) edge[bend right=50] (b);
				\path[draw, dashed, line width=0.8pt] (z) edge[bend left=30] (x);
				\end{tikzpicture}\caption*{(b) $D_{\mathcal{Q}}\left\langle  V_2\right\rangle$}\end{minipage}}
			\subfigure{\begin{minipage}[t]{0.3\linewidth}
					\centering\begin{tikzpicture}[scale=0.7]	
				\filldraw[black](3,1) circle (3pt)node[label=right:$b$](b){};
					\filldraw[black](3,-1.5) circle (3pt)node[label=right:$y$](y){};
					\filldraw[black](4,-4) circle (3pt)node[label=right:{$x$}](x){};
					\filldraw[black](2,-4) circle (3pt)node[label=left:$z$](z){};			
					\foreach \i/\j/\b in {
						x/y/0,
						y/z/0,
						z/x/10,
						x/z/10,
						b/y/0,
						b/z/10
					}{\path[draw, line width=0.8pt] (\i) edge[bend right=\b] (\j);}
					\path[draw, line width=0.8pt] (b) edge[bend left=10] (x);
					\path[draw, dashed, line width=0.8pt] (y) edge[bend left=50] (b);
					\path[draw, dashed, line width=0.8pt] (y) edge[bend right=50] (b);
					\path[draw, dashed, line width=0.8pt] (z) edge[bend right=20] (y);
					\end{tikzpicture}\caption*{(c) $D_{\mathcal{Q^{\prime}}}\left\langle  V_2\right\rangle$}\end{minipage}}
	\caption{An illustration of a 2-feasible set $\mathcal{Q}=(Q_1,Q_2,Q_3)$ such that $W_{\mathcal{Q}}^+=\emptyset$, see (a), where $Q_1=yvb, Q_2=yub$ and $Q_3=zt^{\prime}x$. For readability we have only shown the relevant arcs of $D$. The directed multigraph in (b) is $D_{\mathcal{Q}}\left\langle  V_2\right\rangle$. It can be checked easily that the arc $xy$ is a cut-arc of $D_{\mathcal{Q}}\left\langle  V_2\right\rangle$. The set $\mathcal{Q^{\prime}}=(Q_1,Q_2,Q_3^{\prime})$ with $Q_3^{\prime}=zt^{\prime}y$ is also 2-feasible and the directed multigraph $D_{\mathcal{Q^{\prime}}}\left\langle  V_2\right\rangle$ in (c) is 2-arc-strong.}
	\label{fig-lem31}
\end{figure}

 Now observe that there exists a $2$-feasible set $\mathcal{Q}$ in $D$. In fact, every set of two arc-disjoint $(X,Y)$-paths in $D$ clearly is a $2$-feasible set. Such a pair of paths exist as $D$ is 2-arc-strong. We may assume that \yw{for every 2-feasible set $\mathcal{Q}$,} $D_{\mathcal{Q}}\left\langle  V_{2}\right\rangle$ is not 2-arc-strong since otherwise $D$ has a strong arc decomposition by Corollary \ref{coro-DV2star2ArcStr} with $D^{\ast}=D_{\mathcal{Q}}$. By \yw{the last line of the statement of} Lemma \ref{lem-cutarc}, either $W_{\mathcal{Q}}^+\neq \emptyset$ or $W_{\mathcal{Q}}^-\neq \emptyset$ for each 2-feasible set $\mathcal{Q}$.
\begin{equation}
\mbox{We choose a 2-feasible set } \mathcal{Q}=(Q_1,\ldots,Q_{\alpha}) (\alpha\geq 2) \mbox{ such that } |W_{\mathcal{Q}}^+|+|W_{\mathcal{Q}}^-| \mbox{ is minimum. }\tag{$C1$}
\end{equation} 

\medskip

Recall that $D_{\mathcal{Q}}\left\langle  V_{2}\right\rangle$ is obtained from $D\left\langle  V_{2}\right\rangle$ by adding splitting arcs. So $d_{D_{\mathcal{Q}}\left\langle  V_2\right\rangle}^+(x)\geq d_{D\left\langle  V_2\right\rangle}^+(x)$ and then every vertex of $W_{\mathcal{Q}}^+$ has out-degree at most one in $D\left\langle  V_2\right\rangle$. Moreover, we have the following.

\begin{lem}\label{lem-tx}
	Let $x$ be a vertex of $W_{\mathcal{Q}}^+$. Suppose that $x$ has out-degree exactly one in $D\left\langle  V_2\right\rangle$. Then 
	
	(i) $A(\mathcal{Q})$ contains no arc from $x$ to $V_1$;
	
	(ii) each out-neighbor $t$ of $x$ in $V_1$ belongs to two paths in $\mathcal{Q}$;

	 (iii) if $t$ is an out-neighbor of $x$ in $V_1$ and $Q_i=wtz$ with $i\notin[2]$, then % $X\neq\{w\}$ and 
	 $w$ has out-degree at most two in $D_{\mathcal{Q}}\left\langle  V_2\right\rangle$. In particular, $w$ has out-degree at most one in $D\left\langle  V_2\right\rangle$.
\end{lem}
\begin{proof}
	First we claim that $A(\mathcal{Q})$ contains no arcs from $x$ to $V_1$. If not, say $xty$ is a subpath of some path $Q_j\in\mathcal{Q}$. Then $xy$ is a new out-arc of $x$ in  $D_{\mathcal{Q}}\left\langle  V_2\right\rangle$ and thus $d_{D_{\mathcal{Q}}\left\langle  V_2\right\rangle}^+(x)>d_{D\left\langle  V_2\right\rangle}^+(x)=1$. This contradicts the definition of set $W_{\mathcal{Q}}^+$ and so the claim holds. 
	
	Let $t$ be any out-neighbor of $x$ in $V_1$. Clearly $xt\notin A(\mathcal{Q})$ by (i). As $\mathcal{Q}$ is $2$-feasible, $t$ belongs to at most two paths in $\mathcal{Q}$. Suppose to the contrary that $t$ belongs to at most one path in $\mathcal{Q}$. Since $t$ has out-degree at least three, there is an out-arc $ty$ of $t$ in $D-x$ with $ty\notin A(\mathcal{Q})$. Then $\mathcal{Q}\cup\{xty\}$ is a $2$-feasible set $\mathcal{Q}^{\prime}$. Further, $x\notin W_{\mathcal{Q^{\prime}}}^+$ as $x$ has one more out-arc, i.e.,  $xy$, in $D_{\mathcal{Q^{\prime}}}\left\langle  V_2\right\rangle$, which contradicts (C1). This proves (ii).
	
	Suppose that $Q_i=wtz$ with $i\notin [2]$ such that $t$ is an out-neighbor of $x$ in $V_1$ and $w$ has out-degree at least three in $D_{\mathcal{Q}}\left\langle  V_2\right\rangle$. By (ii), $t$ belongs to only one path in $\mathcal{Q}-Q_i$. Since $t$ has out-degree at least three, there exists an out-neighbor $t^+$ of $t$ such that $t^+\neq x$ and $tt^+\notin A(\mathcal{Q}-Q_i)$. In particular, let $t^+=z$ when $z\neq x$. Let $\mathcal{Q^{\prime}}=(\mathcal{Q}-wtz)\cup \{xtt^+\}$.  Then  $\mathcal{Q^{\prime}}$ is $2$-feasible and both $w$ and $x$ have out-degree at least two in $D_{\mathcal{Q^{\prime}}}\left\langle  V_2\right\rangle$. It should be mentioned that when $x=z$, although $x$ is loosing one in-neighbor, $x\notin W_{\mathcal{Q}}^-$ as $d_{D\left\langle  V_2\right\rangle}^-(x)\geq |V_2|-1\geq 3$. This contradicts (C1) again as $|W_{\mathcal{Q^{\prime}}}^+|<|W_{\mathcal{Q}}^+|$ and $|W_{\mathcal{Q^{\prime}}}^-|=|W_{\mathcal{Q}}^-|$. So $w$ has at most two out-arcs in $D_{\mathcal{Q}}\left\langle  V_2\right\rangle$ and then it has out-degree at most one in $D\left\langle  V_2\right\rangle$ as $wz$ is a splitting arc of $D_{\mathcal{Q}}\left\langle  V_2\right\rangle$. 
\end{proof}

Next we show that the minimum out-degree of $D_{\mathcal{Q}}\left\langle  V_2\right\rangle$ is at least two except for one case in which we have complete knowledge of the structure of the terminal component of $D\left\langle  V_2\right\rangle$.

\begin{lem}\label{lem-Skcycle}
	If $W_{\mathcal{Q}}^+$ is not empty for each 2-feasible set $\mathcal{Q}$, then $S_k$ is a 3-cycle and each vertex of $S_k$ has exactly one out-neighbor in $V_1$. Furthermore, they all have the same out-neighbor in $V_1$.
\end{lem}
\begin{proof}
	Fix an arbitrary vertex $x\in W_{\mathcal{Q}}^+$.  If $|X|=1$, then the vertex of $X$ clearly has out-degree at least two in $D_{\mathcal{Q}}\left\langle  V_{2} \right\rangle$ as $Q_1,Q_2$ are two arc-disjoint paths starting from $X$. So we may assume that $X\neq\{x\}$. Then $x$ has out-degree at least one in $D\left\langle  V_2\right\rangle$ since otherwise we would have $\{x\}=V(S_k)=X$. On the other hand, since every vertex of $W_{\mathcal{Q}}^+$ has out-degree at most one in $D\left\langle  V_2\right\rangle$, we obtain that $x$ has only one out-arc, say $xy$, in $D\left\langle  V_2\right\rangle$. 
	
	Fix an arbitrary out-neighbor $t$ of $x$ in $V_1$. Such a neighbor exists as $D$ is 2-arc-strong. Moreover, $xt\notin A(\mathcal{Q})$ as $A(\mathcal{Q})$ contains no arc from $x$ to $V_1$ by Lemma \ref{lem-tx} (i).  Next we divide the proof into the following three cases.

	\begin{case}
		$|V(S_k)|\leq 2$.
	\end{case}

	In this case, we have $X=V(S_k)$ and either $V(S_{k-1})=\{x\}, V(S_k)=\{y\}$ or $S_k=xyx$. In both cases it follows from the definition of $Q_1$ and $Q_2$ that arc $xy$ is not in $Q_1\cup Q_2$. Since $A(\mathcal{Q})$ contains no arc from $x$ to $V_1$,  the paths $Q_1, Q_2$ are two arc-disjoint $(\{y\},Y)$-paths.

	Next we show that $t\in V(Q_1)\cap V(Q_2)$. Suppose that $t$ belongs to a path $Q_i\in\mathcal{Q}$ with $i\notin [2]$, say $Q_i=wtz$. By Lemma \ref{lem-tx} (iii), $w$ has out-degree at most one in $D\left\langle  V_2\right\rangle$. Observe that $x,y$ are the only two vertices of $V_2$ with out-degree at most one in $D\left\langle  V_2\right\rangle$ as every other vertex dominates both $x$ and $y$. Since $xt\notin A(\mathcal{Q})$, we have $w=y$. Then $Q_1,Q_2$ and $Q_i$ are all starting from $y$, which implies that $y$ has out-degree at least three in $D_{\mathcal{Q}}\left\langle  V_2\right\rangle$. This contradicts Lemma \ref{lem-tx} (iii) and now it follows by Lemma \ref{lem-tx} (ii) that $t\in V(Q_1)\cap V(Q_2)$.

		 Since $Q_1$ and $Q_2$ are arc-disjoint $(\{y\},Y)$-paths, we may assume that $t$ is not the second vertex of $Q_1$. Let $t^-$ be the predecessor of $t$ on the path $Q_1$.  Replacing the arc $t^-t$  with $t^-xt$ in  $\mathcal{Q}$ (see Figure \ref{fig-WQ1} (b)), we obtain a new $2$-feasible set $\mathcal{Q^{\prime}}$ containing $xt$. This means that the vertex $x$ gains one more out-neighbor and thus $x\notin W_{\mathcal{Q^{\prime}}}^+$. Moreover, although $t^-$ is loosing  one out-neighbor after the replacement, $t^-\notin W_{\mathcal{Q^{\prime}}}^+$ as it has at least two out-neighbors $x,y$ in $V_2$, consequently, in $D_{\mathcal{Q^{\prime}}}\left\langle  V_2\right\rangle$. Therefore, $|W_{\mathcal{Q^{\prime}}}^+|<|W_{\mathcal{Q}}^+|=|\{x\}|$ and $|W_{\mathcal{Q^{\prime}}}^-|=|W_{\mathcal{Q}}^-|$, which contradicts (C1).  It should be noted that $Q_1$ and $Q_2$ do not contain any arc from $V_1$ to $x$, so they still are two paths after replacing the arc $t^-t$  with $t^-xt$.
	\begin{figure}[H]
		\centering
	\subfigure{\begin{minipage}[t]{0.45\linewidth}
			\centering\begin{tikzpicture}[scale=0.8]	
			\coordinate [label=center:$V_2$] () at (2,1.8);
			\coordinate [label=center:$V_1$] () at (-1,1.8);
			\draw[line width=1.5pt] (3,0) -- (3,-2); 
			
			\filldraw[black](2,1) circle (3pt)node[](v1){};
			\filldraw[black](2,0) circle (3pt)node[](v2){};
			\filldraw[black](2,-1) circle (3pt)node[label=right:$t^-$](tt){};
			\filldraw[black](2,-2) circle (3pt)node[label=right:$x$](x){};
			\filldraw[black](2,-3) circle (3pt)node[label=right:$y$](y){};
			
			\filldraw[black](-1,0.2) circle (3pt)node[label=left:$u$](u){};
			\filldraw[black](-1,-1) circle (3pt)node[label=left:$t$](t){};
			\filldraw[black](-1,-2.2) circle (3pt)node[label=left:$w$](w){};
			
			\foreach \i/\j/\w/\b in {
				v1/v2/0.8/0,
				v2/tt/0.8/0,
				tt/x/0.8/0,
				x/y/0.8/0,
				y/w/1.5/0,
				w/tt/1.5/0,
				tt/t/1.5/0,
				t/v1/1.5/0,
				x/t/0.8/0,
				t/x/0.8/15}{\path[draw, line width=\w pt] (\i) edge[bend right=\b] (\j);}
			\path[draw, dashed, line width=0.8 pt] (y) edge[bend left=15] (t);
			\foreach \i/\j in {
				t/v2,
				v2/u,
				u/v1}{\path[draw, dashed, line width=0.8 pt] (\i) edge[] (\j);}		
			\end{tikzpicture}\caption*{(a)}\end{minipage}}
	\subfigure{\begin{minipage}[t]{0.45\linewidth}
			\centering\begin{tikzpicture}[scale=0.8]	
			\coordinate [label=center:$V_2$] () at (2,1.8);
			\coordinate [label=center:$V_1$] () at (-1,1.8);
			\draw[line width=1.5pt] (3,0) -- (3,-2); 
			
			\filldraw[black](2,1) circle (3pt)node[](v1){};
			\filldraw[black](2,0) circle (3pt)node[](v2){};
			\filldraw[black](2,-1) circle (3pt)node[label=right:$t^-$](tt){};
			\filldraw[black](2,-2) circle (3pt)node[label=right:$x$](x){};
			\filldraw[black](2,-3) circle (3pt)node[label=right:$y$](y){};
			
			\filldraw[black](-1,0.2) circle (3pt)node[label=left:$u$](u){};
			\filldraw[black](-1,-1) circle (3pt)node[label=left:$t$](t){};
			\filldraw[black](-1,-2.2) circle (3pt)node[label=left:$w$](w){};
			
			\foreach \i/\j/\w/\b in {
				v1/v2/0.8/0,
				v2/tt/0.8/0,
				tt/x/1.5/0,
				x/y/0.8/0,
				y/w/1.5/0,
				w/tt/1.5/0,
				tt/t/0.8/0,
				t/v1/1.5/0,
				x/t/1.5/0,
				t/x/0.8/15}{\path[draw, line width=\w pt] (\i) edge[bend right=\b] (\j);}
			\path[draw, dashed, line width=0.8 pt] (y) edge[bend left=15] (t);
			\foreach \i/\j in {
				t/v2,
				v2/u,
				u/v1}{\path[draw, dashed, line width=0.8 pt] (\i) edge[] (\j);}	
			\end{tikzpicture}\caption*{(b)}\end{minipage}}
		\caption{An illustration of Case 1 in Lemma \ref{lem-Skcycle}. All arcs in $D\left\langle  V_2\right\rangle$ not shown go from top to bottom. For readability we have only shown the relevant arcs of $D$. The dashed and bold paths shown in (a)  and in (b) form a $2$-feasible set $\mathcal{Q}$ and $\mathcal{Q^{\prime}}$ of $D$, respectively.}
		\label{fig-WQ1}
	\end{figure}
	
	This concludes Case 1 and therefore, we may assume that $|V(S_k)|\geq 3$. Recall that $xy$ is the only out-arc of $x$ in $D\left\langle  V_2\right\rangle$. By Observation  \ref{obse-smallvertex}, we have $x,y\in V(S_k)$ and $xy$ is a cut-arc of $S_k$. 
	
	\begin{case}
		$|V(S_k)|\geq 4$.
	\end{case}

Recall that by Theorem \ref{thm-nicedecom}, $V(S_k)$ has a nice decomposition $(U_1,\ldots, U_l)$ with cut-arcs $\{x_iy_i:i\in[r]\}$.  It follows by Observation \ref{obse-smallvtx-in-gd} and the fact that $U_l=X\neq \{x\}$ that $xy=x_2y_2, U_l=\{x_1\}$ and  $U_{l-1}=\{y_1\}=\{x_2\}$. Hence $Q_1$ and $Q_2$ are two arc-disjoint $(\{x_1\},Y)$-paths. Moreover, by the same argument as in the second paragraph of Case 1 (replace $y$ by $x_1$), we have that the vertex $t$, which we fixed above, belongs to both $Q_1$ and $Q_2$. 
	\begin{figure}[H]
			\centering
		\subfigure{\begin{minipage}[t]{0.45\linewidth}
				\centering\begin{tikzpicture}[scale=0.8]	
				\coordinate [label=center:$V_2$] () at (2,2);
				\coordinate [label=center:$V_1$] () at (-1,2);
				\draw[line width=1.5pt] (3,0) -- (3,-2); 
				
				\filldraw[black](2,1) circle (3pt)node[label=right:$b$](b){};
				\filldraw[black](2,0) circle (3pt)node[label=right:$v$](v){};
				\filldraw[black](2,-1) circle (3pt)node[label=right:{$y$}](y){};
				\filldraw[black](2,-2) circle (3pt)node[label=right:{$x$}](x){};
				\filldraw[black](2,-3) circle (3pt)node[label=right:$x_1$](x1){};
				
				\filldraw[black](-1,0.5) circle (3pt)node[label=left:$a$](a){};
				\filldraw[black](-1,-1) circle (3pt)node[label=left:$t$](t){};

				\foreach \i/\j/\w/\b in {
					x1/x/1.5/0,
					x/y/1.5/0,
					y/a/1.5/0,
					a/v/1.5/10,
					v/t/1.5/0,
					t/b/1.5/10,
					t/x/0.8/10,
					x/t/0.8/10}{\path[draw, line width=\w pt] (\i) edge[bend left=\b] (\j);}
					\foreach \i/\j in {
					t/y,
					y/v,
				v/b}{\path[draw, dashed, line width=0.8 pt] (\i) edge[] (\j);}	
			\path[draw, dashed, line width=0.8 pt] (x1) edge[bend left=20] (t);	
				\end{tikzpicture}\caption*{(a)}\end{minipage}}
			\subfigure{\begin{minipage}[t]{0.45\linewidth}
					\centering\begin{tikzpicture}[scale=0.8]	
					\coordinate [label=center:$V_2$] () at (2,2);
					\coordinate [label=center:$V_1$] () at (-1,2);
					\draw[line width=1.5pt] (3,0) -- (3,-2); 
					
					\filldraw[black](2,1) circle (3pt)node[label=right:$b$](b){};
					\filldraw[black](2,0) circle (3pt)node[label=right:$v$](v){};
					\filldraw[black](2,-1) circle (3pt)node[label=right:{$y$}](y){};
					\filldraw[black](2,-2) circle (3pt)node[label=right:{$x$}](x){};
					\filldraw[black](2,-3) circle (3pt)node[label=right:$x_1$](x1){};
					
					\filldraw[black](-1,0.5) circle (3pt)node[label=left:$a$](a){};
					\filldraw[black](-1,-1) circle (3pt)node[label=left:$t$](t){};

					\foreach \i/\j/\w/\b in {
						x1/x/1.5/0,
						x/y/0.8/0,
						y/a/1.5/0,
						a/v/1.5/10,
						v/t/0.8/0,
						t/b/1.5/10,
						t/x/0.8/10,
						x/t/1.5/10}{\path[draw, line width=\w pt] (\i) edge[bend left=\b] (\j);}
					\foreach \i/\j in {
						t/y,
						y/v,
						v/b}{\path[draw, dashed, line width=0.8 pt] (\i) edge[] (\j);}	
					\path[draw, dashed, line width=0.8 pt] (x1) edge[bend left=20] (t);
					\end{tikzpicture}\caption*{(b)}\end{minipage}}
		\caption{An illustration of Case 2 in Lemma \ref{lem-Skcycle} with $x=x_2=y_1,y=y_2=u$. All arcs in $D\left\langle  V_2\right\rangle$ not shown are from top to bottom. For readability we have only shown the relevant arcs of $D$. The dashed and bold paths $Q_1,Q_2$ shown in (a) form a $2$-feasible set $\mathcal{Q}$ of $D$. The two bold paths $Q_2^{\prime}=x_1xtb$ and $Q_3=uav$ in (b) are obtained from $Q_2$. Note that $Q_2^{\prime}$ is an $(X,Y)$-path and then $\mathcal{Q^{\prime}}=(Q_1,Q_2^{\prime},Q_3)$ is a $2$-feasible set.}
		\label{fig-WQ}
	\end{figure}
	
	Observe that if one of the paths $Q_1,Q_2$, say $Q_i$, contains an arc from $V_1$ to $x$, then $Q_i$ must also contain the only out-arc $xy$ of $x$ in $D\left\langle  V_2\right\rangle$ and $x_1x\notin Q_{3-i}$ as $A(\mathcal{Q})$ contains no arc from $x$ to $V_1$. So we may assume w.l.o.g that $x_1x\notin Q_1$ and that $Q_1$ has no arcs from $V_1$ to $x$. Let $Q_2^{\prime}=x_1xt\cup Q_2[t,b]$, where $b$ is the last vertex of $Q_2$. Clearly $Q_2^{\prime}$ is also an $(\{x_1\},Y)$-path. Replacing $Q_2$ with $Q_2^{\prime}$ in $\mathcal{Q}$ and adding all subpaths  $uav$ of $Q_2[x_1,t^-]$ with $a\in V_1$ (see Figure \ref{fig-WQ} (b)), we obtain a $2$-feasible set $\mathcal{Q^{\prime}}$ containing the arc $xt$.  This means that the vertex $x$ gains one more out-neighbor and then $x\notin W_{\mathcal{Q^{\prime}}}^+$. Moreover, although $t^-$ is loosing one out-neighbor after the replacement, $t^-\notin W_{\mathcal{Q^{\prime}}}^+$ as it has at least two out-neighbors $x,x_1$ in $V_2$, consequently, in $D_{\mathcal{Q^{\prime}}}\left\langle  V_2\right\rangle$. Therefore, $|W_{\mathcal{Q^{\prime}}}^+|<|W_{\mathcal{Q}}^+|=|\{x\}|$ and $|W_{\mathcal{Q^{\prime}}}^-|=|W_{\mathcal{Q}}^-|$, which contradicts (C1).
	
		\begin{case}
		$|V(S_k)|=3$.
	\end{case}

	  Observe that in this case every vertex of $S_k$ has out-degree at least one in $D\left\langle  V_2\right\rangle$ as $S_k$ is strong. First we claim that the vertex $t$ that we fixed above belongs to $V(Q_1)\cup V(Q_2)$. If not, then Lemma \ref{lem-tx} (ii)-(iii) shows that there are two paths $Q_i=w_1tz_1, Q_j=w_2tz_2$ with $i,j\geq 3$ containing $t$. Moreover, each of $w_1$ and $w_2$ has out-degree at most two in $D_{\mathcal{Q}}\left\langle  V_2\right\rangle$. This and the fact that $d_{D\left\langle  V_2\right\rangle}^+(w)\geq 1$ for all $w\in V(S_k)$ imply that $w_1\neq w_2$, $V(S_k)=\{x,w_1,w_2\}$ and both $Q_1$ and $Q_2$ are starting from $x$, which contradicts Lemma \ref{lem-tx} (i). Thus $t\in V(Q_1)\cup V(Q_2)$.

	  Next we show that $t\in V(Q_1)\cap V(Q_2)$. Suppose to the contrary that $t\in V(Q_1)-V(Q_2)$. Let $a_1, a_2$ be the first vertex of $Q_1$ and $Q_2$, respectively. If $a_1= a_2$, then construct a new $2$-feasible set $\mathcal{Q^{\prime}}$ by replacing $Q_1$ by $xtQ_1[t,b]$, where $b$ is the last vertex of $Q_1$. It is not difficult to check that both $x$ and $a_1$ are not vertices in  $W_{\mathcal{Q^{\prime}}}^+$, which contradicts (C1). So we may assume that $a_1\neq a_2$ and thus $V(S_k)=\{x,a_1,a_2\}$. On the other hand, by Lemma \ref{lem-tx} (ii), there is a path $Q_i=wtz$ ($i\geq 3$) containing $t$ such that $w$ has out-degree at most two in $D_{\mathcal{Q}}\left\langle  V_2\right\rangle$. Then $w$ must be $x$, which contradicts Lemma \ref{lem-tx} (i). Therefore, $t$ belongs to $Q_1$ and $Q_2$.
	  
	  Now we claim that $t$ is the second vertex of $Q_i$ for all $i\in[2]$. Otherwise construct a new $2$-feasible set $\mathcal{Q^{\prime}}$ from $\mathcal{Q}$ by replacing $Q_i$ with $xtQ_i[t,b]$ and all subpaths $uav$ with $a\in V_1$ of $Q_i[a_i,t]$, we obtain a contradiction with  (C1) as $x\notin W_{\mathcal{Q^{\prime}}}^+$. This shows that $a_1\neq a_2$, $V(S_k)=\{x,a_1,a_2\}$ and there is an arc from each vertex of $S_k$ to $t$. Note that for any vertices $u,v\in V(S_k)$, the arcs $ut,vt$ can be the first arcs of the $(X,Y)$-paths in any given 2-feasible set $\mathcal{F}$. As $W_{\mathcal{F}}^+$ is not empty, the vertex in $V(S_k)-\{u,v\}$ must belong to $W_{\mathcal{F}}^+$. In other words, every vertex in $V(S_k)$ can be regarded as $x$ and then every vertex of $S_k$ has out-degree one in $S_k$, that is, $S_k$ is a 3-cycle.  This completes the proof of Lemma \ref{lem-Skcycle}.
\end{proof}

By symmetry, we have the following. 
\begin{lem}\label{lem-S1cycle}
	If $W_{\mathcal{Q}}^-$ is not empty for each 2-feasible set $\mathcal{Q}$, then $S_1$ is a 3-cycle and each vertex of $S_1$ has exactly one in-neighbor in $V_1$. Furthermore, they all have the same in-neighbor in $V_1$.
\end{lem}

\bigskip
\noindent {\bf Completing the proof of Theorem \ref{thm-Arcdecsplit}.}
\medskip

Let $\mathcal{Q}$ be a $2$-feasible set satisfying (C1). Recall that either $W_{\mathcal{Q}}^+\neq \emptyset$ or $W_{\mathcal{Q}}^-\neq \emptyset$. By reversing all arcs of $D$ if necessary, we may assume that $W_{\mathcal{Q}}^+\neq \emptyset$. 

It follows by Lemma \ref{lem-Skcycle} that the terminal component $S_k$ of $D\left\langle  V_{2}\right\rangle$ is a 3-cycle. Let $S_k=u_1u_2u_3u_1$.  Moreover, by Lemma \ref{lem-Skcycle}, every vertex of $V(S_k)$ has exactly one out-neighbor in $V_1$ and the out-neighbor of them in $V_1$ is the same vertex. Suppose that $t$ is the common out-neighbor of vertices $u_1, u_2$ and $u_3$ in $V_1$. Recall that $t$ has out-degree at least three and $Q_1, Q_2$ are two $(V(S_k),Y)$-paths. As $S_k$ is a 3-cycle, the vertices $u_1,u_2$ and $u_3$ are symmetrical. Relabeling $u_1,u_2$ and $u_3$ if necessary, we may assume that $u_it\in Q_i$ for all $i\in[2]$ and $t$ has an out-arc $tz\in D-u_3$ which is not in $Q_1\cup Q_2$. Note that $u_3$ is the only vertex of $V_2$ with out-degree exactly one in $D_{\mathcal{Q}}\left\langle  V_{2}\right\rangle$, that is, $W_{\mathcal{Q}}^+=\{u_3\}$. By symmetry, if $W_{\mathcal{Q}}^-\neq \emptyset$, then let $S_1=v_1v_2v_3v_1$. Moreover, we may assume that $t^{\prime}v_i$ $(t^{\prime}\in V_1)$ is the last arc of $Q_i$ for $i\in[2]$ and $t^{\prime}$ has an in-neighbor in $wt^{\prime}$ in $D-v_3$. Clearly, $W_{\mathcal{Q}}^-=\{v_3\}$.

Let $\mathcal{Q^{\prime}}=\mathcal{Q}\cup\{u_3tz\}$ if $W_{\mathcal{Q}}^-= \emptyset$ and let $\mathcal{Q^{\prime}}=\mathcal{Q}\cup\{u_3tz\}\cup\{wt^{\prime}v_3\}$ if we also have $W_{\mathcal{Q}}^-\neq \emptyset$, see Figure \ref{fig-proof} (b). It should be mentioned that $\mathcal{Q^{\prime}}$ is not $2$-feasible as $t$ (resp., $t^{\prime}$) is in three paths of $\mathcal{Q^{\prime}}$. Moreover, each vertex of $V_1$, except for $t$ and $t^{\prime}$, belongs to at most two paths in $\mathcal{Q^{\prime}}$. Let $D_{\mathcal{Q^{\prime}}}$ be the directed multigraph obtained from $D$ by splitting off $\mathcal{Q^{\prime}}$. Now we claim that $W_{\mathcal{Q^{\prime}}}^+=\emptyset$. In fact, note that each $u_i\in V(S_k)$ has out-degree exactly two in $D_{\mathcal{Q^{\prime}}}\left\langle  V_{2}\right\rangle$ and every vertex of $V_2-V(S_k)$ has out-degree at least 3 in $D\left\langle  V_{2}\right\rangle$ as it dominates all vertices of $V(S_k)$. Therefore, $W_{\mathcal{Q^{\prime}}}^+=\emptyset$ and we have $W_{\mathcal{Q^{\prime}}}^-=\emptyset$ in the same way. 
	\begin{figure}[H]
		\centering
			\subfigure{\begin{minipage}[t]{0.45\linewidth}
				\centering\begin{tikzpicture}[scale=0.8]	
				\coordinate [label=center:$V_2$] () at (2.5,4);
				\coordinate [label=center:$V_1$] () at (-0.5,4);
				\draw[line width=1.5pt] (4.5,1) -- (4.5,-1); 
				
				\filldraw[black](1.5,3) circle (3pt)node[label=left:$v_1$](v1){};
				\filldraw[black](3.5,3) circle (3pt)node[label=right:$v_2$](v2){};
				\filldraw[black](2.5,2) circle (3pt)node[label=right:{$v_3$}](v3){};
				\filldraw[black](1.5,0) circle (3pt)node[label=above:$w_1$](w1){};
				\filldraw[black](2.5,0) circle (3pt)node[label=above:$w_2$](w2){};
				\filldraw[black](3.5,0) circle (3pt)node[label=above:$w_3$](w3){};				
			\filldraw[black](1.5,-3) circle (3pt)node[label=below:$u_1$](u1){};
		\filldraw[black](3.5,-3) circle (3pt)node[label=below:$u_2$](u2){};
	\filldraw[black](2.5,-2) circle (3pt)node[label=right:{$u_3$}](u3){};

\filldraw[black](-0.5,2) circle (3pt)node[label=left:$t^{\prime}$](tt){};
\filldraw[black](-0.5,-2) circle (3pt)node[label=left:$t$](t){};
\foreach \i/\j/\b in {
	v1/v2/0,
	v2/v3/0,
	v3/v1/0,
	u1/u2/0,
	u2/u3/0,
	u3/u1/0,
	w1/w2/0,
	w2/w3/0,
	u1/t/0,
	u2/t/0,
	u3/t/0,
	t/u1/20,
	t/w2/0,
	t/w3/10,
	w2/tt/0,
	w3/tt/10,
	tt/v1/0,
	tt/v2/0,
	tt/v3/0
}{\path[draw, line width=0.8pt] (\i) edge[bend right=\b] (\j);}
\path[draw, line width=0.8pt] (w1) edge[bend left=10] (tt);
\path[draw, line width=0.8pt] (w3) edge[bend left=25] (w1);
				\end{tikzpicture}\caption*{(a) $D$}\end{minipage}}
		\subfigure{\begin{minipage}[t]{0.45\linewidth}
				\centering\begin{tikzpicture}[scale=0.8]	
				%\coordinate [label=center:$V_2$] () at (2.5,4);
				
				\draw[line width=1.5pt] (4.5,1) -- (4.5,-1); 
			
				\filldraw[black](1.5,3) circle (3pt)node[label=left:$v_1$](v1){};
				\filldraw[black](3.5,3) circle (3pt)node[label=right:$v_2$](v2){};
				\filldraw[black](2.5,2) circle (3pt)node[label=right:{$v_3$}](v3){};
				\filldraw[black](1.5,0) circle (3pt)node[label=left:$w_1$](w1){};
				\filldraw[black](2.5,0) circle (3pt)node[label=above:$w_2$](w2){};
				\filldraw[black](3.5,0) circle (3pt)node[label=right:$w_3$](w3){};				
				\filldraw[black](1.5,-3) circle (3pt)node[label=below:$u_1$](u1){};
				\filldraw[black](3.5,-3) circle (3pt)node[label=below:$u_2$](u2){};
				\filldraw[black](2.5,-2) circle (3pt)node[label=right:{$u_3$}](u3){};

				\foreach \i/\j/\b in {
					v1/v2/0,
					v2/v3/0,
					v3/v1/0,
					u1/u2/0,
					u2/u3/0,
					u3/u1/0,
					w1/w2/0,
					w2/w3/0		
				}{\path[draw, line width=0.8pt] (\i) edge[bend right=\b] (\j);}
				
				\path[draw, line width=0.8pt] (w3) edge[bend left=30] (w1);
					\path[draw, dashed, line width=0.8pt] (u3) edge[bend left=25] (u1);
						\path[draw, dashed, line width=0.8pt] (w2) edge[bend left=25] (v1);
					\foreach \i/\j in {
						u2/w3,
						w3/v2,
						u1/w2,
						w1/v3}{\path[draw, dashed, line width=0.8 pt] (\i) edge[] (\j);}
				\end{tikzpicture}\caption*{(b) $D_{\mathcal{Q^{\prime}}}\left\langle  V_{2}\right\rangle$}\end{minipage}}
		\caption{An illustration of a set $\mathcal{Q^{\prime}}=(Q_1,Q_2,Q_3,Q_4)$ which is not $2$-feasible, where $Q_1=u_1tw_2t^{\prime}v_1, Q_2=u_2tw_3t^{\prime}v_2, Q_3=u_3tz$ with $z=u_1$ and $Q_4=wt^{\prime}v_3$ with $w=w_1$. Both $S_1$ and $S_k$ are 3-cycles and all arcs not shown are from top to bottom. The 2-arc-strong directed multigraph in (b) is obtained from $D\left\langle  V_{2}\right\rangle$ by splitting off all paths in $\mathcal{Q^{\prime}}$. The dashed arcs in (b) are splitting arcs.}
		\label{fig-proof}
	\end{figure}

	By Lemma \ref{lem-cutarc} and the fact that the two $(X,Y)$-paths in $\mathcal{Q^{\prime}}$, i.e., $Q_1,Q_2$, are starting from and ending at two distinct vertices, we have that $D_{\mathcal{Q^{\prime}}}\left\langle  V_{2}\right\rangle$ is 2-arc-strong. It follows by Theorem \ref{thm-SDM} that either $D_{\mathcal{Q^{\prime}}}\left\langle  V_{2}\right\rangle$ has a strong arc decomposition, or $D_{\mathcal{Q^{\prime}}}\left\langle  V_{2}\right\rangle$ is isomorphic to one of the directed multigraphs shown in Figure \ref{fig-DMwithoutAD}. If the former case holds, \yw{repeating the proof of Lemma  \ref{lem-Dstar} with $D^{\ast}=D_{\mathcal{Q^{\prime}}}$ one may obtain a strong arc decomposition of $D$. Here it should be noted that the condition `splitting off at most two pairs at every vertex in $V_1$' in the statement of  Lemma \ref{lem-Dstar} only be used for vertices in $V(D_i)-V(D_{3-i})$ with $i\in[2]$. Therefore, although we splitting three pairs at $t$ and $t^{\prime}$ and splitting at most two pairs at every vertex in $V_1-\{t,t^{\prime}\}$, as $t,t^{\prime}\in V(D_i)\cap V(D_{3-i})$, we also can obtain a strong arc decomposition of $D$ in the same way mentioned in the proof of Lemma \ref{lem-Dstar}.} If the latter case holds,  $D$ also has a strong arc decomposition by Lemma \ref{lem-D2starsamll}, which completes the proof of Theorem \ref{thm-Arcdecsplit}. $\square$

\bigskip

\yw{Observe that in the proof of Theorem \ref{thm-Arcdecsplit}, if a given 2-feasible set $\mathcal{Q}$ violates the condition (C1), then we can construct a new 2-feasible set $\mathcal{Q^{\prime}}$ with  $|W_{\mathcal{Q^{\prime}}}^+|+|W_{\mathcal{Q^{\prime}}}^-|< |W_{\mathcal{Q}}^+|+|W_{\mathcal{Q}}^-|$ from $\mathcal{Q}$ by  a slight modification of $\mathcal{Q}$. This means that we may get a 2-feasible set satisfying (C1) in polynomial time and combing this observation with the fact that all other steps in the proof are constructive, we conclude  the following holds.}
  \begin{thm}\label{cor-polalg}
    There exists a polynomial algorithm which given a 2-arc-strong split digraph $D=(V_1,V_2;A)$ such that each vertex of $V_1$ has minimum in- and out-degree at least 3, finds a strong arc decomposition of $D$.
    \end{thm}

\section{Arc-disjoint out- and in-branchings in split digraphs}\label{sec-GP}

Recall the definitions of in-branchings and out-branchings from the introduction. A relaxation of the problem of whether a given digraph $D$ has a strong arc decomposition is to ask whether it contains an out-branching and an  in-branching which are arc-disjoint. Such pair will be called \textbf{a good pair} in below and more precisely we call it a \textbf{good $(u,v)$-pair} if the roots $u$ and $v$, of the out- and in-branchings respectively, are specified.

Clearly, a strong digraph contains an out-branching (resp., an in-branching) with arbitrary given root. Thus the following simple observation holds.

\begin{obse}\label{ob-Arcdecom-GP}
	Every digraph with a strong arc decomposition contains a good $(u,v)$-pair for every choice of $u,v$.
\end{obse} 

The above observation and Theorem \ref{thm-Arcdecsplit} imply the following.

\begin{thm}\label{thm-Gpairsplit}
	Let $D=(V_1,V_2;A)$ be a 2-arc-strong split digraph such that $V_1$ is independent and  $D\left\langle  V_{2}\right\rangle$ is semicomplete. Suppose that every vertex of $V_1$ has both in- and out-degree at least 3 in $D$. Then $D$ has a good $(u,v)$-pair for every choice of $u,v$.
\end{thm}

\begin{coro}\label{coro-3arcstr}
	Every 3-arc-strong split digraph contains a good $(u,v)$-pair for every choice of $u,v$.
\end{coro}

A split digraph $D=(V_1,V_2;A)$ is \textbf{semicomplete} if every vertex in the independent set, i.e., $V_1$, is adjacent to every vertex in $V_2$. Below we prove that arc-connectivity 2 is sufficient to guarantee a good pair with the same root $u$, that is, a good $(u,u)$-pair, in semicomplete split digraphs.

\begin{lem} \label{(D-X)->Dpair}\cite{bangSDCbranchings}
	Let $D$ be a digraph and let $X$ be a subset of $V(D)$ such that every vertex of $D-X$ has both an in-neighbor and an out-neighbor in $X$. If $X$ has a good $(u,v)$-pair then $D$ has a good $(u,v)$-pair. 
\end{lem}

\begin{thm}\label{thm-S}
	Let $D$ be a 2-arc-strong digraph and let $u$ be a vertex of $D$. Suppose that $S$ is a semicomplete induced subdigraph of $D$ containing the vertex $u$. If every vertex in $D-S$ has at least two in- and out-neighbors in $V(S)$, then there is a good $(u,u)$-pair in $D$.
\end{thm}
\begin{proof}
	 Let $N^+_S(u) = A\cup C$ and $N^-_{S}(u)= B\cup C$, where $C$ is the set of vertices that form a 2-cycle with $u$. Possibly $A, B$ or $C$ is empty.  We may assume that  $A\cup B\neq \emptyset$ since otherwise $S$ clearly has a good pair rooted at $s$ and the claim follows by Lemma \ref{(D-X)->Dpair}. By symmetry (reversing all arcs of $D$ if necessary) we may assume that $A\neq \emptyset$.  
	 Let $Ter(A)$ be the terminal component of $D\left\langle  A \right\rangle$ and let  $Ini(B)$ be the initial component of $D\left\langle  B \right\rangle$ when $B\neq \emptyset$.
	 
	 Since $D$ is 2-arc-strong, there are two arc-disjoint $(V(Ter(A)),W)$-paths $P_1,P_2$ in $D$, where $W=V(Ini(B))$ if $B\neq\emptyset$ and $W=\{u\}$ when $B$ is empty.  Suppose that  $p_1$ is the first vertex of $P_1$ and $q_1$ is the first vertex of $P_1$ in $B\cup C\cup\{u\}$. Similarly,  suppose that $p_2$ is the last vertex of $P_2$ in $A\cup C\cup\{u\}$ and $q_2$ is the last vertex of $P_2$. Let $Q_i=P_i[p_i,q_i]$ for all $i\in[2]$. Note that $p_2=q_2=u$ when $B$ is empty. 
	
	By our assumption, for each vertex $w\in D-S$, if it belongs to $Q_2-Q_1$, then it has an out-neighbor $w_o$ in $V(S)$ which is distinct from the successor of $w$ on $Q_2$ and if $w$ belongs to $Q_1-Q_2$, then it has an in-neighbor $w_I$ in $S$ which is distinct from the predecessor of $w$ on $Q_1$. So we can obtain a good $(u,u)$-pair $(O,I)$ in $D\left\langle V(Q_1\cup Q_2\cup S) \right\rangle$ as follows (if $B$ is empty, then ignore the items related to $B$) and then $D$ has a good pair rooted at $u$ by Lemma \ref{(D-X)->Dpair}, see Figure \ref{fig-gp}.  

	 Construct $I$ from $Q_1$ and an in-branching rooted at $p_1$ in $Ter(A)$ by adding arcs $\{ru:r\in B\cup C\}\cup\{rp_1:r\in A-V(Ter(A))-V(Q_1)\}$ and $\{ww_o: w\in V(Q_2-Q_1)\cap V(D-S)\}$.  Moreover, construct $O$ from the path $Q_2$ and an out-branching rooted at $q_2$ in $Ini(B)$ by adding arcs $\{ur:r\in A\cup C\}\cup\{q_2r:r\in B-V(Ini(B))-V(Q_2)\}$ and $\{w_Iw:w\in V(Q_1-Q_2)\cap  V(D-S)\}$.
\end{proof}
	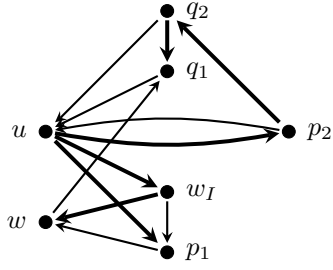
\begin{figure}[H]
		\centering\begin{tikzpicture}[scale=0.8]
		\filldraw[black](2,2) circle (3pt)node[label=right:$q_2$](q2){};
		\filldraw[black](2,1) circle (3pt)node[label=right:$q_1$](q1){};
		\filldraw[black](2,-1) circle (3pt)node[label=right:$w_I$](wI){};
		\filldraw[black](2,-2) circle (3pt)node[label=right:$p_1$](p1){};
		\filldraw[black](4,0) circle (3pt)node[label=right:$p_2$](p2){};
		
		\filldraw[black](0,0) circle (3pt)node[label=left:$u$](u){};
		\filldraw[black](0,-1.5) circle (3pt)node[label=left:$w$](w){};
		
		\foreach \i/\j/\w in {q2/q1/1.5,q2/u/0.8,p2/q2/1.5,q1/u/0.8,w/q1/0.8,u/wI/1.5,u/p1/1.5,wI/p1/0.8,wI/w/1.5,p1/w/0.8}{\path[draw, line width=\w pt] (\i) edge[] (\j);}
		\foreach \i/\j/\w in {u/p2/1.5,p2/u/0.8}{\path[draw, line width=\w pt] (\i) edge[bend right=10] (\j);}
	
		\end{tikzpicture}
	\caption{An illustration of a good $(u,u)$-pair constructed in the proof of Theorem \ref{thm-S}, when $A=\{w_I,p_1\}, B=\{q_1,q_2\}, C=\{p_2\}$ and $V(D-S)=\{w\}$. Moreover, $Q_1=p_1wq_1$ and $Q_2=p_2q_2$. For readability we have only shown the relevant arcs of $D$. The out- and in-branching are shown in bold and thin arcs, respectively.}
	\label{fig-gp}
\end{figure}

The following corollary follows by Theorem \ref{thm-S} with $S=D\left\langle  V_2\cup \{u\} \right\rangle$ and the fact that every vertex in $V_1$ has at least two in-neighbors and at least two out-neighbors in $V_2$ (consequently in $V_2\cup\{u\}$).  
 
\begin{coro}\label{coro-split-ss}
	Let $D=(V_1,V_2;A)$ be a 2-arc-strong split digraph and let $u$ be an arbitrary vertex of $D$. If $D\left\langle  V_2\cup \{u\} \right\rangle$ is semicomplete, then $D$ has a good pair rooted at $u$. In particular, every 2-arc-strong semicomplete split digraph has a good $(u,u)$-pair for every choice of $u$. 
\end{coro}

\begin{conj}
 Every 2-arc-strong semicomplete split digraph $D$ contains a good $(u,v)$-pair for every choice of vertices $u,v$ of $D$.
\end{conj}

The example in Figure \ref{fig:split-nogp} shows that even vertex-connectivity 2 is not sufficient to guarantee the existence of a good $(u,v)$-pair for every choice of $u$ and $v$ in a general split digraph, even when both $u$ and $v$ are vertices of $V_2$ and the subdigraph induced by $V_2$ is strong.
\begin{figure}[H]
	\centering
	\subfigure{\begin{minipage}[t]{0.45\linewidth}
			\centering\begin{tikzpicture}[scale=0.8]
			\coordinate [label=center:$V_2$] () at (0,2.5);
			\coordinate [label=center:$V_1$] () at (-3,2.5);
			\filldraw[black](0,0) circle (3pt)node[label=right:$u_t^+$](ut+){};
			\filldraw[black](0,1) circle (3pt)node[label=right:$v$](v){};
			\filldraw[black](0,2) circle (3pt)node[label=right:$b$](vI+){};
			\filldraw[black](0,-1) circle (3pt)node[label=right:$u$](u){};
			\filldraw[black](0,-2) circle (3pt)node[label=right:$a$](uo-){};
			\filldraw[black](-3,1.5) circle (3pt)node[label=left:$v_t$](vI){};
			\filldraw[black](-3,0) circle (3pt)node[label=left:$w$](w){};
			\filldraw[black](-3,-1.5) circle (3pt)node[label=left:$u_t$](uo){};
			\draw[ line width=0.8pt] (0.5,0) ellipse [x radius=30pt, y radius=20pt];
			\coordinate [label=right:$W$] () at (1.5,0);
			
			\foreach \i/\j/\t in {
				vI+/v/0,
				v/ut+/0,
				ut+/u/0,
				u/uo-/0,
				uo-/uo/0,
				uo-/w/10,
				w/uo-/10,
				u/uo/10,
				uo/u/10,
				uo/ut+/0,
				ut+/vI/0,
				vI/v/10,
				v/vI/10,
				vI+/w/10,
				w/vI+/10,
				vI/vI+/0
			}{\path[draw, line width=0.8] (\i) edge[bend left=\t] (\j);}
			\draw[-stealth,line width=1.8pt] (2.5,1) -- (2.5,-1);	
			\end{tikzpicture}\caption*{$D_1$}\end{minipage}}
	\subfigure{\begin{minipage}[t]{0.45\linewidth}
			\centering\begin{tikzpicture}[scale=0.8]
			\coordinate [label=center:$V_2$] () at (0,2.5);
			\coordinate [label=center:$V_1$] () at (-3,2.5);
			\filldraw[black](0,0) circle (3pt)node[label=right:$u_t^+$](ut+){};
			\filldraw[black](0,1) circle (3pt)node[label=right:$b$](b){};
			\filldraw[black](0,2) circle (3pt)node[label=right:$v$](v){};
			\filldraw[black](0,-1) circle (3pt)node[label=right:$a$](a){};
			\filldraw[black](0,-2) circle (3pt)node[label=right:$u$](u){};
			\filldraw[black](-3,1.5) circle (3pt)node[label=left:$v_t$](vt){};
			\filldraw[black](-3,-1.5) circle (3pt)node[label=left:$u_t$](ut){};
			\draw[ line width=0.8pt] (0.5,0) ellipse [x radius=30pt, y radius=20pt];
			\coordinate [label=right:$W$] () at (1.5,0);
			
			\foreach \i/\j/\t in {
				b/v/0,
				u/a/0,
				a/ut/0,
				ut/u/10,
				u/ut/10,
				ut/ut+/0,
				ut+/vt/0,
				vt/b/0,
				v/vt/10,
				vt/v/10
			}{\path[draw, line width=0.8] (\i) edge[bend left=\t] (\j);}
			\path[draw, line width=0.8] (a) edge[bend right=25] (b);
			\draw[-stealth,line width=1.8pt] (2.5,1) -- (2.5,-1);	
			\end{tikzpicture}\caption*{$D_2$}\end{minipage}}	
	\caption{Split digraphs with vertex partition $V_1\cup V_2$ such that $V_1$ is independent and $D\left\langle V_{2}\right\rangle$ is semicomplete. $W$ is an arbitrary semicomplete digraph containing the vertex $u_t^+$.  All arcs in $D_1\left\langle V_{2}\right\rangle-W$ are from top to bottom and all arcs in $D_2\left\langle V_{2}\right\rangle-W$ except for $\{ua,ab,bv\}$ are from top to bottom. Note that $D_2\left\langle V_{2}\right\rangle$ is strong.}\label{fig:split-nogp}
\end{figure}

\begin{prop}
	The 2-strong split digraphs in Figure \ref{fig:split-nogp} have no good $(u,v)$-pairs.
\end{prop}
\begin{proof}
	It is not difficult to check that $D_i$ is 2-strong for all $i\in[2]$. Now we claim that $uawbv$ and $uu_tu_t^+v_tv$ is the only pair of arc-disjoint $(u,v)$-paths in $D_1$. Note that for any pair of arc-disjoint $(u,v)$-paths $P_1,P_2$ in $D_1$, we can label these so that $awb\in P_1$ and $u_tu_t^+v_t \in P_2$. Since $N_{D_1}^-(a)=\{u,w\}$ and $aw\in P_1$, the arc $ua\in P_1$. In the same way, we have $bv\in P_1$ as $N_{D_1}^+(b)=\{v,w\}$. That is, $P_1=uawbv$. On the other hand, since the out-neighbor are of $u$ in $D_1$ are $u_t, a$ and the arc $ua$ belongs to $P_1$, we have $uu_t\in P_2$ and we also have $v_tv\in P_2$ by symmetry, which shows that  $P_2=uu_tu_t^+v_tv$. 
	
	Suppose that $D_1$ has a good $(u,v)$-pair $(O,I)$ in $D_1$ and assume that $P_i$ and $P_{3-i}$ ($i\in[2]$) are the $(u,v)$-path in $O$ and $I$, respectively. If $i=1$, then $uu_t\in I$ and $aw\in O$. By the definitions of in- and out-branchings, each vertex in $D_1-u$ has in-degree one in $O$ and each vertex in $D_1-v$ has out-degree one in $I$.  Hence the arc $au_t$ should be used in $O$ to collect $u_t$ and also should be used in $I$ to collect $a$, a contradiction. For the case $i=2$, it is not difficult to check that the arc $v_tb$ should be used both in $O$ and $I$, a contradiction again. Therefore, $D_1$ has no good $(u,v)$-pair. Similarly, $uabv$ and $uu_tu_t^+v_tv$ is the only pair of arc-disjoint $(u,v)$-paths in $D_2$ and there is no good $(u,v)$-pair in $D_2$ as the arc $v_tb$ or $au_t$ should be used both in the out- and in-branchings.
\end{proof}

The following proposition follows immediately by the fact that the subdigraph $W$ in Figure \ref{fig:split-nogp} can be any semicomplete digraph. 

\begin{prop}\label{prop-2strNoGP}
	There are infinitely many 2-strong split digraphs which do not have good $(u,v)$-pairs for some choice of $u, v$.
\end{prop}

Combining Observation \ref{ob-Arcdecom-GP} and Proposition \ref{prop-2strNoGP}, we have the following corollary, which implies that the connectivity condition in Theorem \ref{thm-Arcdecsplit} is best possible in some sense.

\begin{coro}\label{coro-2strNoSAD}
	There are infinitely many 2-strong split digraphs which do not have a strong arc decomposition.
\end{coro}

\section{Remarks and further open problems}\label{sec-problem}

\begin{problem}
	Does all but a finite number of  2-arc-strong semicomplete split digraphs have a strong arc decomposition?
\end{problem}

Note that we allow the independent set $V_1$ of a split digraph to be empty. Thus  every semicomplete digraph is a split digraph and hence $S_4$
is  an exception above.\\

The \textbf{degree} of a vertex $v$ in $D$ is $d(v)=d^+(v)+d^-(v)$.
\begin{problem}
	Does every 2-arc-strong split digraph with minimum degree at least 5 have a strong arc decomposition?
\end{problem}

The \textbf{2-linkage problem} is the following: given a digraph $D=(V,A)$ and four distinct vertices $s_1,s_2,t_1,t_2$; does $D$ have a pair of disjoint paths $P_1,P_2$ such that $P_i$ is an $(s_i,t_i)$-path for $i=1,2$?

The 2-linkage problem can be solved in polynomial time for semicomplete digraphs \cite{bangSJDM5}.

\begin{problem}
  Is there a polynomial algorithm for the 2-linkage problem for split digraphs?
\end{problem}

It was shown in \cite{thomassen1984} that every 5-strong semicomplete digraph $D$ is 2-linked, that is, it has a pair of disjoint paths $P_1,P_2$ such that $P_i$ is an $(s_i,t_i)$-path for $i=1,2$ for all choices of distinct vertices $s_1,s_2,t_1,t_2$. It was shown in \cite{bangADM41} that the bound on the connectivity cannot be lowered.

Using the same approach as on pages 386 and 387 in \cite{bang2009} it is easy to show that every 6-strong semicomplete split digraph is 2-linked.

\begin{problem}
  Is every 6-strong split digraph 2-linked?
\end{problem}

%\bibliography{refs}
\end{document}